%% file: main.tex
\documentclass[preprint, 3p]{elsarticle}

\usepackage[english]{babel}
\usepackage[utf8]{inputenc}
\usepackage{amsthm}
\usepackage{amsmath}
\usepackage{amssymb}
\usepackage{hyperref}
\usepackage{enumerate}
\usepackage{xcolor}
\usepackage{color,soul}
\usepackage{tikz}
\usepackage{placeins}
\usepackage{subfigure} 
\usepackage{booktabs}
\usepackage[ruled,vlined]{algorithm2e}

\journal{TBA}

\newtheorem{theorem}{Theorem}
\newtheorem{proposition}{Proposition}
\newtheorem{lemma}{Lemma}
\newtheorem{corollary}{Corollary}
\newtheorem{definition}{Definition}
\newtheorem{remark}{Remark}

\DeclareMathOperator{\tr}{tr}
\DeclareMathOperator{\spt}{spt}

\begin{document}
	
\begin{frontmatter}
	
\title{Local null controllability of a class of non-Newtonian incompressible viscous fluids}
	
\author[UESPI]{P. Carvalho}
\ead{pitagorascarvalho@gmail.com}

\author[IMEUFF]{J. L\'imaco}
\ead{jlimaco@id.uff.br}

\author[IMEUFF]{D. Menezes}
\ead{denilsonjesus@id.uff.br}

\author[IMEUFF]{Y. Thamsten}
\ead{ythamsten@id.uff.br}

\address[UESPI]{Departamento de Matem\'atica, Universidade Estadual do Piau\'i, Teresina, PI, Brasil}

\address[IMEUFF]{Instituto de Matem\'atica e Estat\'istica, Universidade Federal Fluminense, Niter\'{o}i, RJ, Brasil}

\date{}
	
	
\begin{abstract}
    We investigate the null controllability property of systems that mathematically describe the dynamics of some non-Newtonian incompressible viscous flows. The principal model we study was proposed by O. A. Ladyzhenskaya, although the techniques we develop here apply to other fluids having a shear-dependent viscosity. Taking advantage of the Pontryagin Minimum Principle, we utilize a bootstrapping argument to prove that sufficiently smooth controls to the forced linearized Stokes problem exist, as long as the initial data in turn has enough regularity. From there, we extend the result to the nonlinear problem. As a byproduct, we devise a quasi-Newton algorithm to compute the states and a control, which we prove to converge in an appropriate sense. We finish the work with some numerical experiments.
\end{abstract}	

\begin{keyword}
Null controllability, shear dependent viscosity, nonlinear partial differential equations, non-Newtonian fluids. 
\MSC[2010] 35K55, 76D55, 93B05, 93C10.
\end{keyword}

\end{frontmatter}	

\input{Intro.tex}

\input{Linearized.tex}

\input{NullControl.tex}

\input{Numerical.tex}

\input{Comments.tex}


\bibliographystyle{apalike}
\bibliography{References}


\end{document}

%% file: Intro.tex
\section{Introduction} \label{intro}

Let us fix an integer $N \in \{ 2,3 \},$ and let us take a non-empty, open, connected, and bounded subset $\Omega$ of $\mathbb{R}^N$ with a smooth boundary $\partial \Omega,$ and a real number $T>0.$ Henceforth, we write $Q:= \left]0,T\right[\times \Omega,$ and $\Sigma := \left[0,T\right]\times \partial \Omega.$ In general, we understand all of the derivatives figuring in this work in the distributional sense. 

We interpret the set $\Omega$ as a region occupied by the particles of a fluid with a velocity field $y.$ We represent its pressure by $p,$ whereas  $v$ stands for a distributed control which acts as a forcing term through a given open set $\omega \Subset \Omega.$ We assume $\omega \neq \emptyset.$ The model comprising the subject of the current investigation is the following:
\begin{equation} \label{Model}
\begin{array}{l}
\begin{cases}
\frac{D y}{Dt} - \nabla \cdot \mathcal{T}(y,p) = \chi_\omega v, \text{ in } Q, \\
\nabla \cdot y = 0, \text{ in } Q, \\
y = 0, \text{ on } \Sigma, \\
y(0) = y_0, \text{ in } \Omega.
\end{cases}
\end{array}
\end{equation}
Above, the function $\chi_\omega$ denotes the indicator function of $\omega,$ we define the material derivative as
\begin{equation} \label{MaterialDerivative}
    \frac{Dy}{Dt} := y_t + \left( y\cdot \nabla \right) y,
\end{equation} 
the stress tensor, $\mathcal{T},$ is given by
\begin{equation} \label{StressTensor}
    \mathcal{T}(y,p) := -p I + \nu(\nabla y) \nabla y,\, \nu(\nabla y) := \nu_0 + \nu_1 |\nabla y|^{r} ,
\end{equation}
in such a way that the constitutive law for the deviatoric stress tensor reads as
\begin{equation} \label{ConstitutiveLaw}
    \nu(\nabla y)\nabla y := \left( \nu_0 + \nu_1 |\nabla y|^{r} \right) \nabla y,
\end{equation}
where
\begin{equation*}
    |\nabla y| := \left[ \sum_{i,j=1}^N \left( \partial_j y_i\right)^2 \right]^{1/2}.
\end{equation*}
We remark that the three constants $\nu_0, \nu_1,$ and $r$ appearing above are strictly positive, typically with $\nu_0 \gg \nu_1,$ although this assumption is not necessary in this work. 

Therefore, we are focusing on the class of power-law shear-dependent fluids. Pioneers in the study of the system \eqref{Model}-\eqref{ConstitutiveLaw} were O. A. Ladyzhenskaya and J.-L. Lions, see \cite{ladyzhenskaya1969mathematical,ladyzhenskaya1970new,ladyzhenskaya1970modification,lions1969quelques}. Particularly, let us introduce the usual spaces we use in the mathematical analysis of fluid dynamics, i.e.,
$$
H := \left\{ y \in L^2(\Omega)^N : \nabla \cdot y = 0 \text{ in } \Omega,\ y\cdot n = 0 \text{ on } \partial \Omega  \right\}
$$
and
$$
V := \left\{y \in H^1_0(\Omega)^N : \nabla \cdot y = 0 \text{ in } \Omega \right\},
$$
where $n$ denotes the outward unit normal on $\partial \Omega.$ Then, the results \cite[Chapitre 2, Th\`eoremes 5.1-5.2]{lions1969quelques} (cf. \cite[Chapitre 2, Remarque 5.4]{lions1969quelques}) imply the following:
\begin{theorem} \label{thm:wellPosednessNonlinear} 
Let us suppose that 
$$
r > \frac{N}{2} - 1.
$$
as well as
$$
y_0 \in H \text{ and } \chi_\omega v \in L^{q^\prime}\left(0,T; V^\prime \right),
$$
where
$$
\frac{1}{q} + \frac{1}{q^\prime} = 1, \text{ for } q := r+2.
$$
Then, the problem \eqref{Model}-\eqref{ConstitutiveLaw} admits a unique solution $(y,p)$ such that
$$
y \in L^{r+2}(0,T;V) \cap L^\infty(0,T;H) \text{ and } p \in L^2(Q).
$$
\end{theorem}

For $r=1$ and $N=3,$ the system \eqref{Model}-\eqref{ConstitutiveLaw} is the simple turbulence model of Smagorinsky, see \cite{smagorinsky1963general}. Since then, gradient-dependent (or shear-dependent) viscosity models of incompressible viscous fluids have attracted considerable attention from the mathematical, physical, and engineering communities. Some other works investigating the well-posedness for the model \eqref{Model}-\eqref{ConstitutiveLaw} under consideration are \cite{pokorny1996cauchy,du1991analysis,malek1993non,malek1995existence,malek2001weak}. The paper \cite{layton2016energy} studies the energy dissipation for the Smagorinsky model. For the investigation of some regularity properties of solutions of \eqref{Model}-\eqref{ConstitutiveLaw}, see \cite{beirao2005regularity} and the references therein.

On the one hand, the Navier-Stokes (NS) system of equations (corresponding to formally replacing $\nu_1 = 0$ in (\ref{ConstitutiveLaw})) is deeply relevant, not only in mathematics, but for physics, engineering, and biology, see \cite{tartar2006introduction,ruzicka2000electrorheological}. For standard well-posedness results, which are now classic, see \cite{temam1978navier,boyer2012mathematical}. However, even with a great effort of researchers, among the main longstanding open problems are the questions about global existence or finite-time blow-up of smooth solutions in dimension three of the incompressible Navier-Stokes (or else the Euler) equations. The system \eqref{Model}-\eqref{ConstitutiveLaw} is a generalization of the Navier-Stokes equations. From a practical perspective, as \cite{du1991analysis} points out, every fluid which solutions of NS decently models is at least as accurately described by those of \eqref{Model}-\eqref{ConstitutiveLaw}.

On the other hand, for real-world problems, the advantage of considering the more general fluids of power-law type is not slight. In effect, as \cite{malek1995existence} describes, practitioners employed them to investigate problems in chemical engineering of colloids, suspensions, and polymeric fluids, see \cite{sutterby1965laminar,sutterby1966laminar,turian1969critical,carreau1968rheological,christiansen1973isothermal,powell1944mechanisms,ree1958relaxation}, in ice mechanics and glaciology, see \cite{metzner1956non,van1990new,kjartanson1988creep}, in blood-rheology, see \cite{cho1991effects,cho1989effects,cokelet1963rheology,cross1965rheology,davies1990effects,nakamura1988numerical,steffan1990comparison,el2016effects}, and also in geology, see \cite{malevsky1992strongly}, to name a few instances.

We briefly describe the physical meanings of the constants $\nu_0,\,\nu_1,$ and $r.$ Firstly, $\nu_0$ stands for the kinematic viscosity of the fluid. If the physical variables are nondimensionalized, then $\nu_0^{-1}$ is the Reynolds number of the fluid. Secondly, we can conceive the constants $\nu_1$ and $r$ in light of the kinetic theory of gases and the definition of a Stokesian fluid, see \cite{ladyzhenskaya1970modification,ladyzhenskaya1970new}. For instance, from the point of view of turbulence modeling, we have $\nu_1 = C_0\ell^2,$ where $C_0$ is a model parameter and $\ell \ll 1$ is a mixing length, see \cite{prandtl1952guide}. In the latter perspective, a possible derivation of the model stands on the Boussinesq assumption for the Reynolds stress, further stipulating that the eddy viscosity $\nu_t$ takes the particular form
\begin{equation} \label{eddy}
    \nu_t = \nu_1 |\nabla y|^r,
\end{equation}
see \cite{lesieur1987turbulence,rebollo2014mathematical}. The term $\nu_t$ given by (\ref{eddy}) leads to a stabilizing effect by increasing the viscosity for a corresponding increase in the velocity field gradient, see the discussion in \cite{du1991analysis}; hence, we call these fluids shear-thickening.

From the viewpoint of control theory, \cite{fernandez2004local} establishes the local null controllability for the Navier-Stokes equations under no-slip boundary conditions; later developments worth mentioning are, e.g, \cite{fernandez2006some,coron2009null,carreno2013local,coron2014local}. For the study of the Ladyzhenskaya-Smagorinsky model, see \cite{fernandez2015theoretical}. The paper \cite{micu2018local} deals with a similar one-dimensional problem. Regarding local exact controllability properties for scalar equations having a locally nonlinear diffusion, some advances are \cite{fernandez2017theoretical,chaves2015uniform,liu2012local}. However, although the diffusion coefficients can be functions of the state (in the case of \cite{chaves2015uniform} in a simplified form), the methods used in these works seem not enough to tackle the situation in which these coefficients depend on the gradient of the controlled solution. Furthermore, the assumptions they make rule out more general diffusions with power-law type nonlinearities. In the present work, we can circumvent all of these difficulties. 

The notion of controllability we consider in this paper is defined as follows.
\begin{definition} \label{DefnLocNC}
We say that \eqref{Model}-\eqref{ConstitutiveLaw} is locally null-controllable at time $T>0$ if there exists $\eta>0$ such that, for each $y_0 \in \left[H^5(\Omega)\cap V\right]^N$ satisfying the compatibility conditions $Ay_0,A^2y_0 \in \left[H^1_0(\Omega)\right]^N,$ as well as
$$
\|y_0\|_{H^5(\Omega)^N} < \eta,
$$
we can find $v \in L^2(\left]0,T\right[\times \omega)^N$ for which the corresponding velocity field $y$ of \eqref{Model}-\eqref{ConstitutiveLaw} satisfies
\begin{equation} \label{eq:ControlCondn}
    y(T,x) = 0 \text{ for almost every } x \in \Omega.
\end{equation}
\end{definition}

We now state the main theoretical result we establish in this paper. 
\begin{theorem} \label{MainThm}
Let us suppose $r \in \{1,2\}$ or $r \geqslant 3.$ For each $T>0,$ the system \eqref{Model}-\eqref{ConstitutiveLaw} is locally null-controllable at time $T.$ 
\end{theorem}
\begin{remark}
Although we stated Theorem \ref{thm:wellPosednessNonlinear} in terms of weak solutions, our methodology yields smooth controls and transient trajectories for the nonlinear system \eqref{Model}-\eqref{ConstitutiveLaw}. Namely, we will be able to prove that there is a control parameter $v$ such that 
$$
\rho_4 v,\,(\zeta v)_t,\, \zeta \Delta v,\, \left( \widehat{\zeta} v_t \right)_t,\, \widehat{\zeta}\Delta v_t,\,\widehat{\zeta} D^4 v \in L^2(Q)^N,
$$
with a corresponding trajectory $y$ satisfying
\scriptsize
$$
\begin{cases}
\rho_6\nabla y,\, \rho_7  y_t,\,\rho_7 \Delta y,\,\rho_8\nabla y_t,\,\rho_9y_{tt},\,\rho_9\Delta y_t,\,\rho_{10}\nabla y_{tt},\,\rho_{10}D^3 y_t,\,\rho_9 D^4 y,\, \rho_{11}y_{ttt},\,\rho_{11}\Delta y_{tt} \in L^2(Q)^N,\\
\rho_6 y,\, \rho_7 \nabla y,\, \rho_8y_t,\,\rho_9\Delta y,\,\rho_9 \nabla y_t,\,\rho_9 D^3 y,\,\rho_{10}y_{tt},\,\rho_{10}\Delta y_t,\, \rho_{11}\nabla y_{tt} \in L^\infty(0,T; L^2(\Omega)^N),
\end{cases}
$$
\normalsize
for appropriate time-dependent positive weights $\rho_4,$ $\rho_6,$ $\rho_7,$ $\rho_8,$ $\rho_9,$ $\rho_{10},$ $\rho_{11},$ $\zeta,$ $\widehat{\zeta}$ which blow up exponentially as $t\uparrow T.$ For more details and the proofs, we refer to Sections \ref{sec2} and \ref{sec3}. Of course, there is a trade-off between such regularity and our requirements on the initial datum. We will comment upon questions that are related to this relation on Section \ref{Comments}. 
\end{remark}

We will prove Theorem \ref{MainThm} with the aid of a local inversion-to-the-right theorem. Namely, we will introduce Banach spaces $Y$ and $Z$ (we provide the details in the second subsection of Section \ref{sec3}) as well as a mapping $H: Y \rightarrow Z,$ such that a solution $(y,p,v)$ of the equation
\begin{equation} \label{Inverse}
    H(y,p,v) = (0,y_0),
\end{equation}
for a given initial data $y_0$ meeting the assumptions of Theorem \ref{MainThm}, is a solution of the control problem, i.e., a tuple subject to \eqref{Model}-\eqref{ConstitutiveLaw} and \eqref{eq:ControlCondn}. We will use the inversion theorem to guarantee the existence of a local right inverse of $H.$ For proving that $H$ is well-defined, as well as that it enjoys suitable regularity properties, the key steps are novel high-order weighted energy estimates for a control and the solution of the linearization of the system \eqref{Model}-\eqref{ConstitutiveLaw} around the zero trajectory. 

Taking advantage of the invertibility properties of $DH(0,0,0),$ we construct the following algorithm allowing the computation of a tuple $(y,p,v)$ solving \eqref{Model}-\eqref{ConstitutiveLaw} and \eqref{eq:ControlCondn}.
\begin{algorithm}[!htp]
\SetAlgoLined
\KwResult{The tuple $(y,p,v)$ solving \eqref{Model}-\eqref{ConstitutiveLaw} and \eqref{eq:ControlCondn}.}
Initialize with the error variable $\epsilon,$ the tolerance $\epsilon_0,$ $n=0,$ and an initial guess $(y^0,p^0,v^0).$\\
\While{$\epsilon > \epsilon_0$}{
We let $f$ be the solution of
$$
DH(0,0,0)\cdot f = H(y^n,p^n,v^n) - (0,y_0);
$$\\
We set $(y^{n+1},p^{n+1},v^{n+1}) \gets (y^n,p^n,v^n) - f;$\\
We update $\epsilon,$ say, $\epsilon \gets \|y^{n+1} - y^n \|_{L^2(Q)};$\\
We update $n\gets n+1;$
}
\Return{$(y^n,p^n,v^n)$}
\caption{A quasi-Newton algorithm}
\end{algorithm}
\FloatBarrier
 
The following local convergence result for Algorithm 1 holds.
\begin{theorem} \label{NumThm}
There exist a small enough constant $\eta > 0,$ as well as appropriate Banach spaces $Y$ and $Z,$\footnote{We provide, in the second subsection of Section \ref{sec3}, the explicit definitions of both $Y$ and $Z.$} such that, if $\|y_0\|_{H^5(\Omega)^N} < \eta,$ with $y_0$ satisfying the compatibility conditions of Definition \ref{DefnLocNC}, then it is possible to find $\kappa \in \left]0,1\right[$ with the following property: the relations $(y^0,p^0,v^0) \in Y$ and
$$
\|(y^0,p^0,v^0)-(y,p,v)\|_Y < \kappa,
$$
imply the existence of $\theta \in \left]0,1\right[$ for which
$$
\|(y^{n+1},p^{n+1},v^{n+1}) - (y,p,v)\|_Y \leqslant \theta \|(y^n,p^n,v^n)-(y,p,v)\|_Y,
$$
for all $n\geqslant 0.$ In particular, $(y^n,p^n,v^n) \rightarrow (y,p,v)$ in $Y.$
\end{theorem}

Here, we fix some notations that we will use throughout the whole paper. Firstly, $C$ denotes a generic positive constant that may change from line to line within a sequence of estimates. In general, $C$ depends on $\Omega,$ $\omega,$ $T,$ $\nu_0,$ $\nu_1,$ and $r.$ In case $C$ begins to depend on some additional quantity $a$ (or we want to emphasize some dependence), we write $C=C(a).$ We will also write, for every integer $k\geqslant 0,$
$$
\left|D^k y\right| := \left[\sum_{i=1}^N \sum_{|\alpha|=k} \left(\partial^\alpha y_i\right)^2 \right]^{1/2},
$$
where we used the standard multi-index notation above. We denote the standard norm of $L^2(\Omega)^N$ by $\|\cdot\|.$ Finally, we set $\|D^k y\| := \left\| \left| D^k y\right|\right\|.$ 

We finish this introductory section outlining the structure of the remainder of the work.
\begin{itemize}
    \item In Section 2, we study the linearization of \eqref{Model}-\eqref{ConstitutiveLaw} around the zero trajectory --- it is a forced Stokes system. With the aid of a global Carleman estimate, we can to show that this system is null controllable. Assuming sufficiently regular initial data, we employ a bootstrapping argument to deduce higher regularity for the control, taking advantage of its characterization via Pontryagin's minimum principle. The higher control regularity naturally leads to higher regularity of the velocity field..
    
    \item In Section \ref{sec3}, we use a local inversion-to-the-right theorem for mappings between Banach spaces to show that the model \eqref{Model}-\eqref{ConstitutiveLaw} is locally null controllable. 
    
    \item It is in Section 4 that we prove Theorem \ref{NumThm}. Then, we conduct some numerical experiments to illustrate our theoretical findings.
    
    \item Finally, we conclude the work in Section \ref{Comments} with some comments and perspectives. 

\end{itemize}

%% file: Linearized.tex
\section{Study of the linearized problem} \label{sec2}

\subsection{Some previous results}

Our aim in the present Section is to establish the null controllability of the linear system:
\begin{equation} \label{Linear}
\begin{array}{l}
\begin{cases}
Ly + \nabla p = \chi_\omega v + f, \text{ in } Q, \\
\nabla \cdot y = 0, \text{ in } Q, \\
y = 0, \text{ on } \Sigma, \\
y(0) = y_0, \text{ in } \Omega,
\end{cases}
\end{array}
\end{equation}
In \eqref{Linear}, we have written $Ly := y_t - \nu_0 \Delta y.$ We achieve this result via a suitable Carleman inequality for the adjoint system of (\ref{Linear}); upon writing $L^*\varphi := -\varphi_t - \nu_0 \Delta \varphi,$ it reads
\begin{equation} \label{Adjoint}
\begin{array}{l}
\begin{cases}
L^*\varphi + \nabla \pi = g, \text{ in } Q, \\
\nabla \cdot \varphi = 0, \text{ in } Q, \\
\varphi = 0, \text{ on } \Sigma, \\
\varphi(T) = \varphi^T, \text{ in } \Omega.
\end{cases}
\end{array}
\end{equation}

In the present subsection, we fix notations that we will employ henceforth. Let us consider $\omega_1 \Subset \omega,$ with $\omega_1 \neq \emptyset.$ For the proof of the following lemma, see \cite{fursikov1996controllability}.

\begin{lemma} \label{Fursikov}
There is a function $\eta^0 \in C^2(\overline{\Omega})$ satisfying 
$$
\eta^0 >0 \text{ in } \Omega,\ \eta^0 = 0 \text{ on } \partial\Omega,\ |\nabla \eta^0| > 0 \text{ on } \overline{\Omega \backslash \omega_1}. 
$$
\end{lemma}

We take $l \in C^\infty(\left[0,T\right])$ with 
$$
l(t) \geqslant T^2/4 \text{ on } \left[0,T/2\right],\ l(t) = t(T-t), \text{ on } \left[T/2,T\right].
$$
We define
$$
\overline{\gamma}(x) := e^{\lambda(\eta^0(x) +m\|\eta^0\|_\infty)},
$$
$$
\overline{\alpha}(x) := e^{\frac{5}{4}\lambda m\|\eta^0\|_\infty} - e^{\lambda(\eta^0(x) + m\|\eta^0\|_\infty)},
$$
$$
\gamma_1 := \min_{\overline{\Omega}} \overline{\gamma},\ \gamma_2 := \max_{\overline{\Omega}} \overline{\gamma},
$$
$$
\alpha_1 := \min_{\overline{\Omega}} \overline{\alpha},\ \alpha_2 := \max_{\overline{\Omega}} \overline{\alpha},
$$
and
$$
\gamma := \frac{\overline{\gamma}}{l^4},\ \alpha := \frac{\overline{\alpha}}{l^4}.
$$

\begin{remark} \label{rem:multimin_maj_max}
Given $C>1,$ $m>4,$ there exists $\lambda_0=\lambda_0(m,C)>0$ such that $\alpha_2 \leqslant C\alpha_1,$ for all $\lambda \geqslant \lambda_0.$ 
\end{remark}

For $s,\lambda>0,$ we write
\begin{align*}
    I(s,\lambda,\varphi) :=&\, s^3\lambda^4 \int_Q e^{-2s\alpha}\gamma^3|\varphi|^2d(t,x) + s\lambda^2\int_Q e^{-2s\alpha}\gamma |\nabla \varphi|^2 d(t,x) \\
    &+ s^{-1}\int_Q e^{-2s\alpha}\gamma^{-1}\left(|\varphi_t|^2 + |\Delta \varphi|^2 \right) d(t,x).
\end{align*}

We are ready to recall the Carleman inequality that is the key to study the null controllability of the linear system \eqref{Linear}.
\begin{proposition} \label{Carleman}
There exist positive constants $\widehat{s},$ $\widehat{\lambda}$ and $C$ depending solely on $\Omega$ and $\omega$ for which the relations $g \in L^2(Q)^N,$ $\varphi^T \in H,$ $\lambda \geqslant \widehat{\lambda}$ and $s \geqslant \widehat{s}(T^4 + T^8)$ imply 
$$
\begin{array}{l} \displaystyle
I(s,\lambda,\varphi) \leqslant C(1+T^2)\bigg(s^{\frac{15}{2}}\lambda^{20}\int_Q e^{-4s\alpha_1 +2s\alpha_2}\left(\frac{\gamma_2}{l^4}\right)^{\frac{15}{2}}|g|^2d(t,x) 
\\ \noalign{\smallskip} \displaystyle
\hspace{1.7cm} \ +\ s^{16}\lambda^{40}\int_0^T \int_{\omega_1} e^{-8s\alpha_1 + 6s\alpha_2}\left(\frac{\gamma_2}{l^4}\right)^{16}|\varphi|^2dx\ dt \bigg),
\end{array}
$$
where $\varphi$ is the solution of (\ref{Adjoint}) corresponding to $g$ and $\varphi^T$.
\end{proposition}

As a consequence, we get the following Observability Inequality.
\begin{corollary} \label{Observability}
With the notations of Proposition \ref{Carleman} (possibly enlarging $\widehat{s},$ $\widehat{\lambda}$ and $C,$ the latter now depending on $T$), we have
\scriptsize
$$
\|\varphi(0)\|^2 \leqslant C\left(s^{\frac{15}{2}}\lambda^{20}\int_Q e^{-4s\alpha_1 +2s\alpha_2}\gamma_2^{\frac{15}{2}}|g|^2d(t,x) + s^{16}\lambda^{40}\int_0^T \int_{\omega_1} e^{-8s\alpha_1 + 6s\alpha_2}\gamma_2^{16}|\varphi|^2dx\ dt \right).
$$
\normalsize
\end{corollary}

From now on, we fix $\lambda = \widehat{\lambda}$ and $s=\widehat{s}.$ Moreover, in view of Remark \ref{rem:multimin_maj_max}, given $\gamma > 0,$ we can take $\widehat{\lambda} = \widehat{\lambda}(\gamma)$ large enough in such a way that 
\begin{equation} \label{eq:multmin_maj_max}
    \alpha_2 < (1+\gamma)\alpha_1.
\end{equation}
Whenever we need \eqref{eq:multmin_maj_max} in subsequent estimates, for a suitable positive real number $\gamma,$ we will assume it holds in all that follows.

For $p,q,r \in \mathbb{R},$ we introduce the weights
$$
\mu_{p,q,r}(t):= \exp\left\{ps\alpha_1 l^{-4}(t) \right\}\exp\left\{qs\alpha_2 l^{-4}(t) \right\} l^r(t).
$$
Regarding these weights, it is valuable to note:
\begin{remark} \label{Weights}
Let $p,p_1,p_2,q,q_1,q_2,r,r_1,r_2$ be nine real numbers.
\begin{itemize}
    \item[(a)] One has the equality $\mu_{p_1,q_1,r_1}\mu_{p_2,q_2,r_2} = \mu_{p_1+p_2,q_1+q_2,r_1+r_2}.$ In particular, for integral $k,$ $\mu_{p,q,r}^k = \mu_{kp,kq,kr}.$
    \item[(b)] There exists a constant $C>0$ such that 
    $$
    \begin{cases}
    \left|\frac{d}{dt} \mu_{p,q,r}\right|\leqslant C\mu_{p,q,r-5}, \\
    \left|\frac{d}{dt} \left(\mu_{p,q,r}^2\right) \right|\leqslant C \mu_{p,q,r-\frac{5}{2}}^2. 
    \end{cases}
    $$
    \item[(c)] There exists a constant $C>0$ such that $\mu_{p_1,q_1,r_1} \leqslant C\mu_{p_2,q_2,r_2}$ if, and only if,
    $$
    \begin{cases}
    p_1\alpha_1 + q_1\alpha_2 = p_2\alpha_1 + q_2\alpha_2 \text{ and } r_1 \geqslant r_2, \\
    \text{ or } \\
    p_1\alpha_1 + q_1\alpha_2 < p_2\alpha_1 + q_2\alpha_2.
    \end{cases}
    $$
\end{itemize}
\end{remark}

We define the weights
$$
\rho_0 := \mu_{0,1,6},
$$
$$
\rho_1 := \mu_{0,1,2},
$$
$$
\rho_2 := \mu_{0,1,-2},
$$
$$
\rho_3 := \mu_{2,-1,15},
$$
and
$$
\rho_4 := \mu_{4,-3,32}.
$$
With these notations, we can gather Proposition \ref{Carleman} and Corollary \ref{Observability} together, resulting in the following statement.
\begin{corollary} \label{Carleman2}
There is a constant $C=C(\Omega,\omega,\widehat{s},\widehat{\lambda},m,T)>0$ such that the solution $\varphi$ of (\ref{Adjoint}) corresponding to $g \in L^2(Q)^N$ and $\varphi^T \in H$ satisfies
$$
\begin{array}{l} \displaystyle
\|\varphi(0)\|^2 + \int_{Q}\left[ \rho_0^{-2}|\varphi|^2 + \rho_1^{-2}|\nabla\varphi|^2 + \rho_2^{-2}\left(|\varphi_t|^2 + |\Delta \varphi|^2 \right)\right]d(t,x)
\\ \noalign{\smallskip} \displaystyle
\hspace{2.5cm} \ \leqslant C\left(\int_Q\rho_3^{-2}|g|^2d(t,x) + \int_0^T \int_{\omega_1}\rho_{4}^{-2}|\varphi|^2 dx\ dt \right).
\end{array}
$$
\end{corollary}

\subsection{Null controllability of the linear system}

\begin{theorem} \label{ControlOfLinearSystem}
We suppose $y_0 \in H,$ $\rho_0 f \in L^2(Q)^N.$ Then there exist controls $v \in L^2(\left]0,T\right[\times\omega)^N$ such that the state $y$ of (\ref{Linear}) corresponding to $v, f$ and $y_0$ satisfies
\begin{equation} \label{estimate0}
    \int_Q \rho_3^2|y|^2d(t,x) + \int_0^T\int_\omega \rho_4^2|v|^2dx\ dt \leqslant C\kappa_0(y_0,f),
\end{equation}
where
$$
\kappa_0(y_0,f) := \|y_0\|_H^2 + \int_Q \rho_0^2|f|^2d(t,x).
$$
In particular, $y(T) = 0$ almost everywhere in $\Omega.$
\end{theorem}
\begin{proof}
We define $P_0 := \left\{ (w,\sigma) \in C^2(\overline{Q})^{N+1} : \nabla\cdot w \equiv 0,\ w|_\Sigma \equiv 0,\ \int_\Omega \sigma \ dx = 0 \right\},$ we take $\chi \in C^\infty_c(\omega),$ with $0 \leqslant \chi \leqslant 1,$ $\chi|_{\omega_1}\equiv 1,$ and we consider on $P_0$ the continuous bilinear form
$$
b((w,\sigma),(\widetilde{w},\widetilde{\sigma})) := \int_Q \left\{ \rho_3^{-2}\left(L^*w +\nabla \sigma\right)\cdot \left(L^* \widetilde{w} + \nabla\widetilde{\sigma}\right) + \chi\rho_4^{-2}w\cdot \widetilde{w} \right\} d(t,x).
$$
By Corollary \ref{Carleman2}, $b$ is an inner product on $P_0.$ Let us denote $P:= \overline{P_0}^{b(\cdot,\cdot)},$ i.e., $P$ is the completion of $P_0$ under the norm induced by $b(\cdot,\cdot).$ We also deduce, from the corollary we just mentioned, that that the linear form
$$
\Lambda : (w,\sigma) \in P \longmapsto \int_\Omega y_0\cdot w(0) dx + \int_Q f\cdot w\ d(t,x) \in \mathbb{R}
$$
is continuous, with
$$
|\Lambda(w,\sigma)| \leqslant C\kappa_0(y_0,f)^{1/2}b((w,\sigma),(w,\sigma))^{1/2}.
$$
The Riesz representation theorem guarantees the existence of a unique $(\varphi,\pi) \in P$ for which
\begin{equation} \label{VariationalProblem}
    \Lambda(w,\sigma) = b((w,\sigma),(\varphi,\pi)) \hspace{1.0cm} (\text{for all } (w,\sigma) \in P). 
\end{equation}
Upon taking $(w,\sigma) = (\varphi,\pi)$ above, we get
$$
b((\varphi,\pi),(\varphi,\pi)) = \Lambda(\varphi,\pi) \leqslant C\kappa_0(y_0,f)^{1/2}b((\varphi,\pi),(\varphi,\pi))^{1/2},
$$
whence
$$
b((\varphi,\pi),(\varphi,\pi)) \leqslant C\kappa_0(y_0,f).
$$
Let us set 
\begin{equation} \label{PontryaginMinPrinciple}
    y:= \rho_3^{-2}\left(L^*\varphi + \nabla \pi \right),\ z:= \rho_4^{-2}\varphi,\ v:= -\chi z.
\end{equation}
We observe that $\spt(v) \subseteq \omega,$ that $(y,v)$ is a solution of (\ref{Linear}) corresponding to the datum $y_0$ and $f,$ and applying Corollary \ref{Carleman2} once more,
$$
\int_Q \rho_3^2|y|^2d(t,x) + \int_0^T \int_\omega \rho_4^2|v|^2 dx\ dt \leqslant Cb((\varphi,\pi),(\varphi,\pi)) \leqslant C\kappa_0(y_0,f).
$$
This proves the theorem.
\end{proof}

\subsection{Weighted energy estimates} \label{Additional}

Along this subsection, we let $y_0 \in H,$ $\rho_0 f \in L^2(Q)^N,$ and let us denote by $(v,y)$ the control-state pair constructed in the proof of Theorem \ref{ControlOfLinearSystem}.

\begin{lemma} \label{LemmaAdditional1}
Let us define $\rho_6 := \mu_{1,-1/2,35/2}$ and $\rho_7 := \mu_{1,-1/2,20}.$ We have
\begin{equation} \label{estimate1}
    \sup_{\left[0,T\right]}\left( \int_\Omega \rho_6^2 |y|^2 dx\right) + \int_Q \rho_6^2 |\nabla y|^2 d(t,x) \leqslant C\kappa_0(y_0,f),
\end{equation}
and, if $y_0 \in H^1_0(\Omega)^N,$ then
\begin{equation} \label{estimate2}
    \int_Q \rho_7^2\left(|y_t|^2+|\Delta y|^2 \right)d(t,x) + \sup_{\left[0,T\right]} \left(\int_\Omega \rho_7^2|\nabla y|^2 dx \right) \leqslant C\kappa_1(y_0,f),
\end{equation}
where
$$
\kappa_1(y_0,f) := \|y_0\|_{H^1_0(\Omega)^N}^2 + \int_Q\rho_0^2 |f|^2d(t,x).
$$
\end{lemma}
\begin{proof}
For each $n \geqslant 1,$ let $v_n(t,\,\cdot),\,f_n(t,\,\cdot)$ and $y_{0,n}(\cdot)$ be the projections of of $v(t,\,\cdot),\,f(t,\,\cdot)$ and $y_0(\cdot)$ in the first $n$ eigenfunctions for the Stokes operator $A: D(A) \rightarrow H,$ respectively. Let us denote by $y_n$ the corresponding solution for the finite dimensional approximate forced Stokes system. For simplicity, unless we state otherwise, we omit the subscript $n$ throughout the current proof. Moreover, we emphasize that we can take all of the constants $C$ appearing below to be independent of $n.$

Using $\rho_6^2y$ as a test function in system (\ref{Linear}), and doing some integrations by parts, we derive the identity
\begin{align} \label{1Adt1}
  \begin{split}
    \frac{1}{2}\frac{d}{dt}\left(\int_\Omega \rho_6^2 |y|^2dx \right) + \nu_0 \int_\Omega \rho_6^2 |\nabla y|^2 dx =&\, \int_\omega \rho_6^2 v\cdot y\ dx + \int_\Omega \rho_6^2 f\cdot y\ dx \\
    &+ \frac{1}{2}\int_\Omega\frac{d}{dt}\left(\rho_6^2 \right) |y|^2dx.
  \end{split}
\end{align}
From \eqref{eq:multmin_maj_max} and Remark \ref{Weights}, item (c), we have $\rho_6 \leqslant C\rho_4 \leqslant C\rho_3 \leqslant C\rho_0,$ whence
\begin{equation} \label{1Adt2}
    \int_\omega \rho_6^2 v\cdot y\ dx \leqslant C\left(\int_\omega \rho_4^2 |v|^2 dx + \int_\Omega \rho_3^2 |y|^2 dx \right),
\end{equation}
and
\begin{equation} \label{1Adt3}
    \int_\Omega \rho_6^2 f\cdot y\ dx \leqslant C\left(\int_\Omega \rho_0^2 |f|^2 dx + \int_\Omega\rho_3^2 |y|^2 dx \right).
\end{equation}
From Remark \ref{Weights}, item (b), we have $\left|\frac{d}{dt}(\rho_6^2)\right| \leqslant C\rho_3^2,$ from where it follows that
\begin{equation} \label{1Adt4}
    \int_\Omega \frac{d}{dt}\left( \rho_6^2 \right)|y|^2dx \leqslant C\int_\Omega \rho_3^2|y|^2 dx.
\end{equation}
Using (\ref{1Adt2}), (\ref{1Adt3}) and (\ref{1Adt4}) in (\ref{1Adt1}), and applying Gronwall's inequality together with (\ref{estimate0}), we infer (\ref{estimate1}). 

Henceforth, we will tacitly apply \eqref{eq:multmin_maj_max} and Remark \ref{Weights}.

Now, we use $\rho_7^2(y_t-\nu_0 A y)$ as a test function in (\ref{Linear}), from where we easily derive that
\footnotesize
\begin{equation} \label{1Adt5}
\begin{array}{l} \displaystyle
    \int_\Omega \rho_7^2\left(|y_t|^2 + \nu_0^2|\Delta y|^2 \right)dx + \frac{1}{2}\frac{d}{dt}\left(\int_\Omega \rho_7^2|\nabla y|^2 dx \right) = \int_\omega \rho_7^2 v\cdot(y_t-\nu_0A y)dx  
    \\ \noalign{\smallskip} \displaystyle
    \hspace{4.0cm} \ + \int_\Omega \rho_7^2 f\cdot(y_t-\nu_0A y)dx + \frac{1}{2}\int_\Omega \frac{d}{dt}\left( \rho_7^2 \right)|\nabla y|^2 dx.
\end{array}
\end{equation}
\normalsize
We observe that, for any $\epsilon>0,$
\begin{equation} \label{1Adt6}
    \int_\omega \rho_7^2 v\cdot(y_t-\nu_0A y)dx \leqslant \frac{C}{\epsilon}\int_\omega \rho_4^2|v|^2dx + C\epsilon\left[\int_\Omega \rho_7^2(|y_t|^2 + |\Delta y|^2)dx \right],
\end{equation}
\begin{equation} \label{1Adt7}
    \int_\Omega \rho_7^2 f\cdot(y_t-\nu_0A y)dx \leqslant \frac{C}{\epsilon}\int_\Omega \rho_0^2|f|^2dx + C\epsilon\left[\int_\Omega \rho_7^2(|y_t|^2 + |\Delta y|^2)dx \right],
\end{equation}
\begin{equation} \label{1Adt8}
    \int_\Omega \frac{d}{dt}\left( \rho_7^2\right)|\nabla y|^2 dx \leqslant C\int_\Omega \rho_6^2 |\nabla y|^2 dx.
\end{equation}
We take $\epsilon$ sufficiently small, in such a way that that the terms involving $y$ in (\ref{1Adt6}) and (\ref{1Adt7}) are absorbed by the left-hand side of (\ref{1Adt5}). Also, from (\ref{1Adt8}) and (\ref{estimate1}), the time integral of the third term in the right-hand side of (\ref{1Adt5}) is bounded by $C\kappa_0(y_0,f).$ Thus, it suffices to apply Gronwall's Lemma to conclude (\ref{estimate2}) for the Galerkin approximates $y_n$ instead of the actual solution $y.$ Employing standard limiting arguments, as $n \rightarrow \infty,$ we conclude that \eqref{estimate2} does hold for the actual solution $y.$
\end{proof}

\begin{lemma} \label{RegularControl}
\begin{itemize}
    \item[(a)] If $\zeta := \mu_{-1,1,0},$ then 
    $$
    \zeta v \in L^2(0,T;H^2(\omega)\cap H^1_0(\omega))\cap C(\left[0,T\right];V),\ (\zeta v)_t \in L^2(\left]0,T\right[\times\omega)^N,
    $$
    with the estimate
    $$
    \int_0^T \int_\omega \left[|(\zeta v)_t|^2 + \left|\zeta \Delta v\right|^2 \right]dx\ dt + \sup_{\left[0,T\right]}\|\zeta v\|_V^2 \leqslant C\kappa_0(y_0,f).
    $$
    \item[(b)] Let us also assume that $y_0 \in H^1_0(\Omega)^N.$ For $\widehat{\zeta} := \mu_{-1,1,5},$ we have the memberships
    $$
    (\widehat{\zeta} v_t)_t \in L^2(\left]0,T\right[\times \omega)^N,\ \widehat{\zeta} v_t \in L^2\left(0,T;\left[H^2(\omega)\cap H^1_0(\omega)\right]^N\right),
    $$
    $$
    \widehat{\zeta} v \in L^2\left(0,T;\left[H^4(\omega)\cap H^1_0(\omega)\right]^N\right),
    $$
    and the following inequality holds
    $$
    \int_0^T \int_\omega \left[\left|(\widehat{\zeta} v_t)_t\right|^2 + \left|\widehat{\zeta} \Delta v_t\right|^2 + \left|\widehat{\zeta} \Delta^2v\right|^2 \right] dx\ dt \leqslant C\kappa_1(y_0,f).
    $$
\end{itemize}
\end{lemma}
\begin{proof}
\textit{(a)} For $p,q,r \in \mathbb{R},$ we notice that
$$
L^*(\mu_{p,q,r}z) = -\frac{d}{dt}\left(\mu_{p-8,q+6,r-64}\right)\varphi + \mu_{p-4,q+4,r-34}y - \mu_{p-8,q+6,r-64}\nabla \pi.
$$
Choosing $p=-1,$ $q=1,$ and $r=0,$ it follows that
$$
\left| \frac{d}{dt}\left(\mu_{p-8,q+6,r-64} \right)\right| \leqslant C\rho_0^{-1},\ \mu_{p-4,q+4,r-34} \leqslant C\rho_3,\ \mu_{p-8,q+6,r-64} \leqslant C.
$$
Thus, $u:= \zeta z$ and $\widetilde{\pi} := \mu_{-9,7,-64}\pi$ solve the Stokes equation
$$
\begin{cases}
Lu + \nabla \widetilde{\pi} = h, \text{ in } Q, \\
\nabla\cdot u = 0, \text{ in } Q, \\
u = 0, \text{ on } \Sigma, \\
u(T) = 0, \text{ in } \Omega,
\end{cases}
$$
where 
$$
h := -\frac{d}{dt}\left(\mu_{-9,7,-64}\right)\varphi + \mu_{-5,5,-34}y \in L^2(Q)^N.
$$
By standard regularity results for solutions of the Stokes system, we can infer the stated regularity for $\zeta v = -\chi u.$

\textit{(b)} As in the previous item, for $p,q,r \in \mathbb{R},$ we derive
$$
\begin{array}{l} \displaystyle
    L^*(\mu_{p,q,r}z_t) = -\varphi\frac{d}{dt}\left[\mu_{p,q,r}\frac{d}{dt}(\mu_{-8,6,-64}) \right] + y\mu_{p+4,q-2,r+64}\frac{d}{dt}(\mu_{-8,6,-64})
    \\ \noalign{\smallskip} \displaystyle
    \hspace{2.0cm} \ -\ \frac{d}{dt}(\mu_{p-8,q+6.r-64})\varphi_t + \mu_{p-4,q+4,r}y_t + \mu_{p-8,q+6,r-64}\frac{d}{dt}(\mu_{4,-2,64}) y
    \\ \noalign{\smallskip} \displaystyle
    \hspace{2.0cm} \ -\ \mu_{p-8,q+6,r-64}\nabla \pi_t.
\end{array}
$$
For the choice $p=-1, q=1, r=5,$ it is straightforward to check the inequalities
$$
\left|\frac{d}{dt}\left[\mu_{p,q,r}\frac{d}{dt}(\mu_{-8,6,-64}) \right] \right| \leqslant C\mu_{p-8,q+7,r-58}\rho_0^{-1} = C\mu_{-9,8,-53}\rho_0^{-1} \leqslant C\rho_0^{-1}, 
$$
$$
\left|\mu_{p+4,q-2,r+64}\frac{d}{dt}(\mu_{-8,6,-64}) \right| \leqslant C\mu_{p-14,q+6,r-152}\rho_3 = C\mu_{-15,7,-147}\rho_3 \leqslant C\rho_3,
$$
$$
\left|\frac{d}{dt}(\mu_{p-8,q+7,r-64}) \right|  \leqslant C\mu_{p-8,q+7,r-71}\rho_2 = C\mu_{-9,8,-66}\rho_2 \leqslant C\rho_2,
$$
$$
\mu_{p-4,q+4,r} = \mu_{p-6,q+5,r-20} \rho_7 = \mu_{-7,6,-15}\rho_7 \leqslant C\rho_7,
$$
$$
\left|\mu_{p-8,q+6,r-64}\frac{d}{dt}(\mu_{4,-2,64}) \right| \leqslant C\mu_{p-6,q+5,r-37}\rho_3 = C\mu_{-7,6,-32}\rho_3 \leqslant C\rho_3,
$$
and
$$
\mu_{p-8,q+6,r-64} = \mu_{-9,7,-59} \leqslant C.
$$
We can conclude by arguing similarly as in the first two memberships and corresponding estimates. The third ones are obtained doing the same analysis for the term $L^*(\zeta \Delta z).$
\end{proof}

\begin{lemma}
Let us set $\rho_8 := \zeta = \mu_{-1,1,0}$ and $\rho_9 := \mu_{-1,1,\frac{5}{2}}.$ Supposing $y_0 \in H^2(\Omega)^N\cap V,$ $Ay_0 \in \left[H^1_0(\Omega)\right]^N,$ $\rho_8f_t \in L^2(Q)^N,$ we have the following estimates:
\begin{equation} \label{estimate3}
    \sup_{\left[0,T\right]}\left(\int_\Omega \rho_8^2|y_t|^2dx \right) + \int_Q \rho_8^2|\nabla y_t|^2 d(t,x) \leqslant C\kappa_2(y_0,f).
\end{equation}
If furthermore $y_0 \in H^3(\Omega)^N,$ $f(0) \in H^1_0(\Omega)^N,$
\begin{equation} \label{estimate4}
    \int_Q \rho_9^2\left(|y_{tt}|^2 + |\Delta y_t|^2 \right)d(t,x) + \sup_{\left[0,T\right]}\left[\int_Q \rho_9^2\left(|\nabla y_t|^2 + |\Delta y|^2 \right)dx \right] \leqslant C\kappa_3(y_0,f),
\end{equation}
where
$$
\kappa_2(y_0,f):= \|y_0\|_{H^2(\Omega)^N}^2 + \int_Q\rho_0^2 |f|^2d(t,x) + \int_Q \rho_8^2|f_t|^2 d(t,x)
$$
and
$$
\kappa_3(y_0,f) := \|y_0\|_{H^3(\Omega)^N}^2 + \int_Q\rho_0^2 |f|^2d(t,x) + \int_Q \rho_8^2|f_t|^2 d(t,x) + \|f(0)\|_{H^1_0(\Omega)^N}^2. 
$$
\end{lemma}
\begin{proof}
We establish the current estimates by following the same approach as in the proof of Lemma \ref{LemmaAdditional1}. Here, we begin by differentiating the system (\ref{Linear}) with respect to time, and we use $\rho_8^2 y_t$ as a test function:
\begin{align} \label{2Adt1}
  \begin{split}
    \frac{1}{2}\frac{d}{dt}\left(\int_\Omega \rho_8^2|y_t|^2 dx \right) + \nu_0 \int_\Omega \rho_8^2 |\nabla y_t|^2 dx =&\, \int_\omega \rho_8^2 v_t\cdot y_t dx + \int_\Omega \rho_8^2 f_t\cdot y_t dx \\
    &+ \frac{1}{2}\int_\Omega \frac{d}{dt}(\rho_8^2)|y_t|^2 dx.
  \end{split}
\end{align}
We note that
$$
\left| \frac{d}{dt}(\rho_8^2) \right| \leqslant C\rho_7^2,\ \rho_8 \leqslant C\rho_7 \leqslant C\rho_4;
$$
hence,
\begin{equation} \label{2Adt2}
    \int_\omega \rho_8^2 v_t\cdot y_t dx \leqslant C\left[\int_\omega\left(|\rho_4 v|^2 + |(\zeta v)_t|^2 \right)dx + \int_\Omega \rho_7^2 |y_t|^2 dx  \right],
\end{equation}
\begin{equation} \label{2Adt3}
    \int_\Omega \rho_8^2 f_t \cdot y_t dx \leqslant C\left(\int_\Omega \rho_8^2 |f_t|^2 dx + \int_\Omega \rho_7^2|y_t|^2 dx \right).
\end{equation}
By Lemmas \ref{LemmaAdditional1} and \ref{RegularControl},
$$
\int_Q\rho_7^2 |y_t|^2  d(t,x) + \int_0^T\int_\omega\left(|\rho_4 v|^2 + |(\zeta v)_t|^2 \right)dx\ dt \leqslant \kappa_1(y_0,f),
$$
so that by using (\ref{2Adt2}) and (\ref{2Adt3}) in (\ref{2Adt1}), then integrating in time and applying of Gronwall's lemma, it follows that
\small
\begin{align} \label{2Adt4}
  \begin{split}
    \sup_{\left[0,T\right]}\left(\int_\Omega \rho_8^2|y_t|^2 dx\right) + \int_Q \rho_8^2|\nabla y_t|^2 d(t,x) \leqslant C&\left(\|y_t(0)\|_{L^2(Q)^N}^2 +\int_Q \rho_8^2|f_t|^2d(t,x) \right.\\
    &\hspace{1.0cm} + \kappa_1(y_0,f) \Big).
  \end{split}
\end{align}
\normalsize
It is simple to infer the subsequent estimate:
\begin{equation} \label{2Adt5}
    \|y_t(0)\|_{L^2(Q)^N}^2 \leqslant \|y_0\|_{H^2(\Omega)^N}^2 + \|f(0)\|_{L^2(Q)^N}^2 + \|v(0)\|_{L^2(\left]0,T\right[\times\omega)^N}^2 \leqslant \kappa_2(y_0,f).
\end{equation}
Relations (\ref{2Adt4}) and (\ref{2Adt5}) imply (\ref{estimate3}).

Next, we use $\rho_9^2\left(y_{tt}-\nu_0A y_t\right)$ as a test function in the system (\ref{Linear}) differentiated with respect to time to deduce
\footnotesize
\begin{equation} \label{2Adt6}
    \begin{array}{l} \displaystyle
    \int_\Omega \rho_9^2 \left(|y_{tt}|^2 + \nu_0^2|\Delta y_t|^2 \right)dx + \nu_0\frac{d}{dt}\left( \int_\Omega \rho_9^2 |\nabla y_t|^2 dx \right) = \int_\omega \rho_9^2 v_t\cdot (y_{tt}-\nu_0A y_t) dx 
    \\ \noalign{\smallskip} \displaystyle
    \hspace{4.2cm} \ + \int_\Omega \rho_9^2 f_t\cdot (y_{tt}-\nu_0A y_t)dx + \nu_0\int_\Omega \frac{d}{dt}\left(\rho_9^2 \right)|\nabla y_t|^2 dx 
    \end{array}
\end{equation}
\normalsize
We observe that 
\begin{equation} \label{2Adt7}
    \int_\Omega \frac{d}{dt}(\rho_9^2) |\nabla y_t|^2dx \leqslant C\int_\Omega \rho_8^2|\nabla y_t|^2 dx,
\end{equation}
and for each $\epsilon > 0,$
\begin{equation} \label{2Adt8}
    \int_\omega \rho_9^2 v_t\cdot (y_{tt}-\nu_0A y_t)dx \leqslant C\left[ \frac{1}{\epsilon} \int_\omega \zeta^2 |v_t|^2 dx + \epsilon \int_\Omega \rho_9^2\left( |y_{tt}|^2 + |\Delta y_t|^2 \right) dx \right],
\end{equation}
as well as
\begin{equation} \label{2Adt9}
    \int_\Omega \rho_9^2f_t \cdot (y_{tt}-\nu_0A y_t)dx \leqslant C\left[\frac{1}{\epsilon}\int_\Omega \rho_8^2|f_t|^2 dx + \epsilon\int_\Omega \rho_9^2(|y_{tt}|^2 + |\Delta y_t|^2)dx \right].
\end{equation}
We fix a sufficiently small $\epsilon,$ whence the second terms within the brackets in the right-hand sides of (\ref{2Adt8}) and (\ref{2Adt9}) are absorbed by the left-hand side of (\ref{2Adt6}). Then, using (\ref{2Adt7}) in (\ref{2Adt6}), we infer
\footnotesize
\begin{equation} \label{2Adt10}
    \begin{array}{l} \displaystyle
    \int_\Omega \rho_9^2 \left(|y_{tt}|^2 + |\Delta y_t|^2 \right)dx + \frac{d}{dt}\left( \int_\Omega \rho_9^2 |\nabla y_t|^2 dx \right) \leqslant C\bigg(\int_\omega \zeta^2 |v_t|^2 dx + \int_\Omega \rho_8^2 |f_t|^2dx
    \\ \noalign{\smallskip} \displaystyle
    \hspace{8.0cm} \ + \int_\Omega \rho_8^2 |\nabla y_t|^2 dx \bigg).
    \end{array}
\end{equation}
\normalsize
Employing Gronwall's lemma in (\ref{2Adt10}), we obtain
\footnotesize
\begin{equation} \label{2Adt11}
    \int_Q \rho_9^2 (|y_{tt}|^2 + |\Delta y_t|^2)d(t,x) + \sup_{\left[0,T\right]}\left(\int_\Omega \rho_9^2|\nabla y_t|^2 dx \right) \leqslant C\left(\|\nabla y_t(0)\|_{L^2(Q)^N}^2 + \kappa_2(y_0,f) \right).
\end{equation}
\normalsize
We easily establish, with the aid of item (a) of Lemma \ref{RegularControl}, that
\small
$$
\|\nabla y_t(0)\|_{L^2(Q)^N}^2 \leqslant C\left(\|y(0)\|_{H^3(Q)^N}^2 + \|\nabla v(0)\|_{L^2(Q)^N}^2 + \|\nabla f(0)\|_{L^2(Q)^N}^2 \right) \leqslant C\kappa_3(y_0,f), 
$$
\normalsize
whence
\begin{equation} \label{2Adt11AndAHalf}
    \int_Q \rho_9^2 (|y_{tt}|^2 + |\Delta y_t|^2)d(t,x) + \sup_{\left[0,T\right]}\left(\int_\Omega \rho_9^2|\nabla y_t|^2 dx \right) \leqslant C\kappa_3(y_0,f).
\end{equation}

Finally, we use $\rho_9^2\Delta y_t$ in the undifferentiated partial differential equation of system (\ref{Linear}) as a test function to get
\footnotesize
\begin{equation} \label{2Adt12}
    \begin{array}{l} \displaystyle
    \int_\Omega \rho_9^2 |\nabla y_t|^2 dx + \frac{\nu_0}{2}\frac{d}{dt}\left(\int_\Omega \rho_9^2 |\Delta y|^2 dx \right) 
    \\ \noalign{\smallskip} \displaystyle
    \hspace{2.0cm} \ = \int_\omega \rho_9^2 v\cdot \Delta y_t dx + \int_\Omega \rho_9^2 f\cdot \Delta y_t dx + \frac{\nu_0}{2}\int_\Omega \frac{d}{dt}(\rho_9^2)|\Delta y|^2 dx 
    \\ \noalign{\smallskip} \displaystyle
    \hspace{2.0cm} \ \leqslant C\left(\int_\omega \rho_4^2 |v|^2 dx + \int_\Omega \rho_0^2 |f|^2 dx + \int_\Omega \rho_9^2 |\Delta y_t|^2 dx + \int_\Omega \rho_7^2|\Delta y|^2 dx\right).
    \end{array}
\end{equation}
\normalsize
We use Gronwall's lemma in (\ref{2Adt12}), in such a way that
\begin{equation} \label{2Adt13}
    \sup_{\left[0,T\right]} \left(\int_\Omega \rho_9^2 |\Delta y|^2 dx \right) \leqslant C\kappa_3(y_0,f).    
\end{equation}

From the estimates (\ref{2Adt11AndAHalf}) and (\ref{2Adt13}), together with the compatibility condition $Ay_0 \in \left[H^1_0(\Omega)\right]^N,$ we derive (\ref{estimate4}).
\end{proof}

\begin{lemma} \label{LemmaAdditional3}
We write $\rho_{10}:= \widehat{\zeta} = \mu_{-1,1,5}$ and $\rho_{11} := \mu_{-1,1,15/2}.$ Let us assume that $y_0 \in H^4(\Omega)^N\cap V,$ $Ay_0,\,A^2y_0 \in \left[H^1_0(\Omega)\right]^N,$ $\rho_9 \Delta f \in L^2(Q)^N,$ $\rho_{10} f_{tt} \in L^2(Q)^N,$ $\rho_{10}f_t \in L^2\left(0,T; H^1_0(\Omega)^N\right),$ $f(0) \in \left[H^2(\Omega)\cap H^1_0(\Omega)\right]^N$ and $f_t(0)\in L^2(\Omega).$ Then, the following estimate holds
\small
\begin{equation} \label{estimate5}
    \begin{array}{l}\displaystyle
        \sup_{\left[0,T\right]}\left[\int_\Omega \rho_{10}^2 \left(|y_{tt}|^2 + |\Delta y_t|^2 \right)dx + \|\rho_9 y\|_{H^3(\Omega)^N}^2 \right] 
        \\ \noalign{\smallskip} \displaystyle
        \hspace{1.7cm} \ + \int_Q  \left(\rho_{10}^2 |\nabla y_{tt}|^2 +\rho_{10}^2|D^3 y_t|^2 + \rho_9^2|D^4 y|^2 \right) d(t,x) \leqslant C\kappa_4(y_0,f).
    \end{array}
\end{equation}
\normalsize
If, furthermore, $y_0 \in H^5(\Omega)^N,$ $f(0) \in H^3(\Omega)^N,$ $Af(0) \in V,$ and $f_t(0) \in H^1_0(\Omega)^N,$ then
\begin{equation} \label{estimate6}
    \sup_{\left[0,T\right]}\left(\int_\Omega \rho_{11}^2 |\nabla y_{tt}|^2 dx \right) + \int_Q\rho_{11}^2\left(|y_{ttt}|^2 + |\Delta y_{tt}|^2 \right)d(t,x) \leqslant C \kappa_5(y_0,f).
\end{equation}
where we have written
\begin{align*}
\kappa_4(y_0,f) :=& \int_Q\left(\rho_9^2 |\Delta f|^2 +\rho_{10}^2|\nabla f_t|^2 + \rho_{10}^2|f_{tt}|^2\right) d(t,x) \\
&+ \|y_0\|_{H^4(\Omega)^N}^2 + \|f(0)\|_{H^2(\Omega)^N}^2 + \|f_t(0)\|_{L^2(\Omega)^N}^2 + \kappa_3(y_0,f),
\end{align*}
$$
\kappa_5(y_0,f) := \|y_0\|_{H^5(\Omega)^N}^2 + \|f(0)\|_{H^3(\Omega)^N}^2 +\|f_t(0)\|_{H^1_0(\Omega)^N}^2 + \kappa_4(y_0,f).
$$
\end{lemma}
\begin{proof}
Again, we proceed in the same framework as in the proof of Lemma \ref{LemmaAdditional1}. We begin by applying the Stokes operator $A$ on the equation of system (\ref{Linear}), and then use $-\rho_9^2A^2 y$ as a test function:
\footnotesize
\begin{equation} \label{3Adt1}
    \begin{array}{l} \displaystyle
    \nu_0\int_\Omega \rho_9^2|\Delta^2 y|^2 dx = -\int_\omega \rho_9^2 \Delta v \cdot A^2 y\ dx - \int_\Omega \Delta f \cdot A^2 y\ dx + \int_\Omega \rho_9^2 \Delta y_t \cdot A^2 y\ dx
    \\ \noalign{\smallskip} \displaystyle
    \hspace{1.5cm} \ \leqslant C\left(\int_\omega \zeta^2|\Delta v|^2 dx + \int_\Omega \rho_9^2 |\Delta f|^2 dx + \int_\Omega \rho_9^2 |\Delta y_t|^2 dx \right) + \frac{1}{2}\int_\Omega \rho_9^2 |\Delta^2 y|^2 dx,
    \end{array}
\end{equation}
\normalsize
We integrate (\ref{3Adt1}) with respect to time, whence
\begin{equation} \label{3Adt2}
    \int_Q \rho_9^2 |\Delta^2 y|^2 d(t,x) \leqslant C\kappa_4(y_0,f).
\end{equation}
We can now easily argue that, under suitable limiting arguments (having in view the compatibility conditions we required in the statement of the present lemma), Eq. \eqref{3Adt2} yields the corresponding estimate for the solution of \eqref{Linear}. We observe that the relations $\rho_9A y_t \in L^2(Q)^N$ and $\rho_9A^2 y \in L^2(Q)^N$ imply $\rho_9 y \in L^\infty(0,T; H^3(\Omega)^N),$ with
\begin{equation} \label{3Adt3}
    \sup_{\left[0,T\right]}\|\rho_9 y\|_{H^3(\Omega)^N}^2 \leqslant C\int_Q \rho_9^2 \left( |\Delta y_t|^2 + |\Delta^2 y|^2 \right)d(t,x) \leqslant C\kappa_4(y_0,f).
\end{equation}

In the differential equation of system (\ref{Linear}) differentiated once with respect to time, we use the test function $\rho_{10}^2 A^2 y_{t}:$ 
\begin{align} \label{APosterioriAdded1}
    \begin{split}
        \frac{1}{2}\frac{d}{dt}\left(\int_\Omega \rho_{10}^2 |\Delta y_{t}|^2 dx \right) +&\, \nu_0\int_\Omega \rho_{10}^2 |\nabla \Delta y_{t}|^2 dx \\ =&\, \int_\omega \rho_{10}^2 \nabla v_{t} : \nabla \Delta y_{t} dx + \int_\Omega \rho_{10}^2 \nabla  f_{t} : \nabla \Delta y_{t}dx  \\ &+ \frac{1}{2}\int_\Omega \frac{d}{dt}\left( \rho_{10}^2 \right) |\Delta y_{t}|^2 dx.
    \end{split}
\end{align}

For $\epsilon > 0,$

\small
\begin{equation} \label{APosterioriAdded2}
    \begin{array}{l} \displaystyle
        \int_\omega \rho_{10}^2 \nabla v_{t} : \nabla \Delta y_{t}dx \leqslant C_{\epsilon}\int_\omega \widehat{\zeta}^2 |\nabla v_{t}|^2 dx + \epsilon \int_\Omega \rho_{10}^2|\nabla \Delta y_{t}|^2 dx 
        \\ \noalign{\smallskip} \displaystyle
        \hspace{2.375cm} \ \leqslant C_{\epsilon} \int_\omega\left( \left|(\widehat{\zeta}v_t)_t\right|^2 + |(\zeta v)_t|^2 + \rho_4^2|v|^2\right)dx + \epsilon\int_\Omega \rho_{10}^2 |\nabla \Delta y_{t}|^2 dx,
    \end{array}
\end{equation}
\normalsize
\begin{equation} \label{APosterioriAdded3}
    \begin{array}{l} \displaystyle
         \int_\Omega \rho_{10}^2 \nabla f_{t}: \nabla \Delta y_{t} dx \leqslant C_{\epsilon}\int_\Omega \rho_{10}^2 |f_{tt}|^2 dx + \epsilon\int_\Omega \rho_{10}^2 |\nabla \Delta y_{t}|^2 dx ,
    \end{array}
\end{equation}
\begin{equation} \label{APosterioriAdded4}
    \int_\Omega \frac{d}{dt}\left( \rho_{10}^2\right)|\Delta y_{t}|^2 dx \leqslant C \int_\Omega \rho_9^2 |\Delta y_{t}|^2,
\end{equation}
and
\begin{equation} \label{APosterioriAdded5}
    \begin{array}{l} \displaystyle
        \|\Delta y_{t}(0)\|^2 \leqslant C\left(\|y_0\|_{H^4(\Omega)^N}^2 + \|\Delta v(0)\|_{L^2(\left]0,T\right[\times \omega)^N}^2 + \|\Delta f(0)\|^2 \right)
        \\ \noalign{\smallskip} \displaystyle
        \hspace{1.375cm} \ \leqslant C\kappa_4(y_0,f).
    \end{array}
\end{equation}
Therefore, by taking $\epsilon$ sufficiently small, and using (\ref{APosterioriAdded2})-(\ref{APosterioriAdded5}) in (\ref{APosterioriAdded1}), we deduce
\begin{equation} \label{APosterioriAdded6}
    \sup_{\left[0,T\right]}\left(\int_\Omega \rho_{10}^2 |\Delta y_t|^2dx \right) + \int_Q \rho_{10}^2 |\nabla\Delta y_t|^2 d(t,x) \leqslant C\kappa_4(y_0,f).
\end{equation}

Next, we differentiate the equation of system (\ref{Linear}) twice with respect to time and we use the test function $\rho_{10}^2 y_{tt}:$ 
\begin{align} \label{3Adt4}
  \begin{split}
    \frac{1}{2}\frac{d}{dt}\left(\int_\Omega \rho_{10}^2 |y_{tt}|^2 dx \right) + \nu_0\int_\Omega \rho_{10}^2 |\nabla y_{tt}|^2 dx =&\, \int_\omega \rho_{10}^2 v_{tt}\cdot y_{tt} dx + \int_\Omega \rho_{10}^2 f_{tt}\cdot y_{tt}dx \\
    &+ \frac{1}{2}\int_\Omega \frac{d}{dt}\left( \rho_{10}^2 \right) |y_{tt}|^2 dx.
  \end{split}
\end{align}
We have
\small
\begin{equation} \label{3Adt5}
    \begin{array}{l} \displaystyle
        \int_\omega \rho_{10}^2 v_{tt}\cdot y_{tt}dx \leqslant C\left(\int_\omega \widehat{\zeta}^2 |v_{tt}|^2 dx + \int_\Omega \rho_9^2|y_{tt}|^2 dx \right)
        \\ \noalign{\smallskip} \displaystyle
        \hspace{2.375cm} \ \leqslant C\left[ \int_\omega\left( \left|(\widehat{\zeta}v_t)_t\right|^2 + |(\zeta v)_t|^2 + \rho_4^2|v|^2\right)dx + \int_\Omega \rho_9^2 |y_{tt}|^2 dx\right],
    \end{array}
\end{equation}
\normalsize
\begin{equation} \label{3Adt6}
    \begin{array}{l} \displaystyle
         \int_\Omega \rho_{10}^2 f_{tt}\cdot y_{tt} dx \leqslant C\left(\int_\Omega \rho_{10}^2 |f_{tt}|^2 dx + \int_\Omega \rho_9^2 |y_{tt}|^2 dx \right),
    \end{array}
\end{equation}
\begin{equation} \label{3Adt7}
    \int_\Omega \frac{d}{dt}\left( \rho_{10}^2\right)|y_{tt}|^2 dx \leqslant C \int_\Omega \rho_9^2 |y_{tt}|^2,
\end{equation}
and
\small
\begin{equation} \label{3Adt8}
    \begin{array}{l} \displaystyle
        \|y_{tt}(0)\|^2 \leqslant C\left(\|y_0\|_{H^4(\Omega)^N}^2 + \|\Delta v(0)\|_{L^2(\left]0,T\right[\times \omega)^N}^2 + \|\Delta f(0)\|^2 \right.
        \\ \noalign{\smallskip} \displaystyle
        \hspace{2.0cm} \ \left. + \|v_t(0)\|_{L^2(\left]0,T\right[\times\omega)^N}^2 + \|f_t(0)\|^2 \right)
        \\ \noalign{\smallskip} \displaystyle
        \hspace{1.375cm} \ \leqslant C\kappa_4(y_0,f).
    \end{array}
\end{equation}
\normalsize
Using (\ref{3Adt5}), (\ref{3Adt6}) and (\ref{3Adt7}) in (\ref{3Adt4}), then integrating in time and using (\ref{3Adt8}), we infer
\begin{equation} \label{3Adt9}
    \sup_{\left[0,T\right]}\left(\int_\Omega \rho_{10}^2|y_{tt}|^2 dx \right) + \int_Q \rho_{10}^2 |\nabla y_{tt}|^2 d(t,x) \leqslant C\kappa_4(y_0,f). 
\end{equation}
Estimates (\ref{3Adt2}), (\ref{3Adt3}), (\ref{APosterioriAdded6}) and (\ref{3Adt9}) are enough to conclude (\ref{estimate5}).

Now, we use $\rho_{11}^2(y_{ttt} - \nu_0 A y_{tt})$ as a test function in the equation of system (\ref{Linear}) twice differentiated in time, reaching
\begin{equation} \label{3Adt10}
    \begin{array}{l} \displaystyle
        \int_\Omega \rho_{11}^2|y_{ttt}|^2 dx + \nu_0^2 \int_\Omega \rho_{11}^2|\Delta y_{tt}|^2 dx + \nu_0 \frac{d}{dt}\left(\int_\Omega \rho_{11}^2 |\nabla y_{tt}|^2 dx \right) 
        \\ \noalign{\smallskip} \displaystyle
        \hspace{0.5cm} \ = \int_\omega \rho_{11}^2 v_{tt}\cdot (y_{ttt}-\nu_0A y_{tt}) dx + \int_\Omega \rho_{11}^2 f_{tt}\cdot (y_{ttt}-\nu_0 A y_{tt})dx 
        \\ \noalign{\smallskip} \displaystyle
        \hspace{1.0cm} \ +\, \nu_0 \int_\Omega \frac{d}{dt}(\rho_{11}^2)|\nabla y_{tt}|^2 dx.
    \end{array}
\end{equation}
For $\epsilon > 0,$
\begin{equation} \label{3Adt11}
    \begin{array}{l} \displaystyle
        \int_\omega \rho_{11}^2 v_{tt}\cdot (y_{ttt}-\nu_0A y_{tt}) dx \leqslant C\bigg[\frac{1}{\epsilon}\int_\omega \left(\left|(\widehat{\zeta}v_t)_t\right|^2 + |(\zeta v)_t|^2 + \rho_4^2|v|^2\right)dx 
        \\ \noalign{\smallskip} \displaystyle
        \hspace{5.0cm} \ +\ \epsilon\int_\Omega \rho_{11}^2\left( |y_{ttt}|^2 + |\Delta y_{tt}|^2 \right)dx \bigg],
    \end{array}
\end{equation}
\small
\begin{equation} \label{3Adt12}
    \int_\Omega \rho_{11}^2 f_{tt}\cdot (y_{ttt}-\nu_0 A y_{tt})dx \leqslant C\left[\frac{1}{\epsilon}\int_\Omega \rho_{10}^2 |f_{tt}|^2 dx + \epsilon\int_\Omega \rho_{11}^2\left( |y_{ttt}|^2 + |\Delta y_{tt}|^2 \right)dx\right],
\end{equation}
\normalsize
and we also notice that
\begin{equation} \label{3Adt13}
    \int_\Omega \frac{d}{dt}(\rho_{11}^2)|\nabla y_{tt}|^2 dx \leqslant C \int_\Omega \rho_{10}^2 |\nabla y_{tt}|^2 dx.
\end{equation}
We easily check that
\begin{equation} \label{3Adt14}
    \|\nabla y_{tt}(0)\|^2 \leqslant C\kappa_5(y_0,f).
\end{equation}
As in the proof of (\ref{estimate5}), inequalities (\ref{3Adt11})-(\ref{3Adt14}), we can infer from (\ref{3Adt10}), through an adequate choice of a small positive $\epsilon,$ and the aid of the previous estimates, the subsequent inequality
\begin{equation} \label{3Adt15}
    \int_Q \rho_{11}^2 \left( |y_{ttt}|^2 +|\Delta y_{tt}|^2\right) d(t,x) + \sup_{\left[0,T\right]}\left(\int_\Omega \rho_{11}^2 |\nabla y_{tt}|^2 dx \right) \leqslant C\kappa_5(y_0,f).
\end{equation}
Estimate (\ref{3Adt15}) is precisely (\ref{estimate6}); hence, we have finished the proof of the present result.
\end{proof}

%% file: NullControl.tex
\section{Null controllability of the model (\ref{Model})} \label{sec3}

\subsection{Local right inversion theorem}

It is possible to find a proof of the subsequent result in \cite{alekseev1987optimal}. This is the inversion theorem that we will use to obtain our local null controllability result.
\begin{theorem} \label{Liusternik}
Let $Y$ and $Z$ be two Banach spaces, and $H : Y \rightarrow Z$ be a continuous function, with $H(0) = 0.$ We assume that there are three constants $\delta,\eta^\prime,M > 0$ and a continuous linear mapping $\Lambda$ from $Y$ onto $Z$ with the following properties:
\begin{itemize}
    \item[(i)] For all $e \in Y,$ we have
    $$
    \|e\|_Y \leqslant M\|\Lambda(e)\|_Z;
    $$
    \item[(ii)] The constants $\delta$ and $M$ satisfy $\delta < M^{-1};$
    \item[(iii)] Whenever $e_1,e_2 \in B_Y(0;\eta^\prime),$ the inequality 
    $$
    \|H(e_1) - H(e_2) - \Lambda(e_1-e_2)\|_Z \leqslant \delta \|e_1-e_2\|_Y
    $$
    holds.
\end{itemize}
Then, whenever $k \in B_Z(0;\eta),$ the equation $H(e) = k$ has a solution $e \in B_Y(0;\eta^\prime),$ where $\eta:= \left(M^{-1} - \delta\right)\eta^\prime.$
\end{theorem}

A typical way of verifying condition (iii) is through the remark presented below.
\begin{remark}
Let $Y$ and $Z$ be two Banach spaces, and let us consider a continuous mapping $H:Y\rightarrow Z,$ with $H(0)=0,$ of class $C^1(Y,Z).$ Then it has property $(iii),$ for any positive $\delta$ subject to $(ii),$ as long as we take $\eta^\prime$ as the continuity constant of $DH \in \mathcal{L}(Y,\mathcal{L}(Y,Z))$ at the origin of $Y.$
\end{remark}

\subsection{The setup} \label{setup}
Let us set
$$
X_1 := L^2(Q;\rho_3^2)^N,\ X_2 := L^2(\left]0,T\right[\times \omega; \rho_4^2)^N.
$$
We define 
\begin{equation}
    \begin{array}{l} \displaystyle
        Y := \bigg\{ (y,p,v) \in X_1 \times L^2(Q) \times X_2 : y_t \in L^2(Q)^N,\ \nabla y \in L^2(Q)^{N\times N},
        \\ \noalign{\smallskip} \displaystyle
        \hspace{2.0cm} (\zeta v)_t,\ \zeta\Delta v,\ (\widehat{\zeta} v_t)_t,\ \widehat{\zeta}\Delta v_t,\ \widehat{\zeta}D^4 v \in L^2(\left]0,T\right[\times\omega)^N,
        \\ \noalign{\smallskip} \displaystyle
        \hspace{1.2cm} \text{ for } f:= Ly + \nabla p - \chi_\omega,\ \rho_0 f,\ \rho_8 f_t,\ \rho_9 \Delta f,\ \rho_{10}f_{tt} \in L^2(Q)^N,
        \\ \noalign{\smallskip} \displaystyle
        \hspace{1.5cm} \rho_{10}f_t \in L^2\left(0,T; H^1_0(\Omega)^N\right),\ f(0) \in \left[H^3(\Omega)\cap H^1_0(\Omega)\right]^N, 
        \\ \noalign{\smallskip} \displaystyle
        \hspace{0.9cm} \ Af(0) \in H^1_0(\Omega)^N,\, f_t(0) \in H^1_0(\Omega)^N, \ y|_\Sigma \equiv 0,\ \nabla \cdot y \equiv 0,
        \\ \hspace{0.9cm} \ y(0) \in \left[H^5(\Omega)\cap V\right]^N,\, Ay(0), \, A^2y(0) \in H^1_0(\Omega)^N, \int_\Omega p\ dx = 0  \bigg\}.
    \end{array}
\end{equation}
We consider on $Y$ the norm
\scriptsize
\begin{equation} \label{NormOfY}
    \begin{array}{l} \displaystyle
        \|(y,p,v)\|_Y^2 := \int_Q\left(\rho_3^2|y|^2 + \rho_0^2 |f|^2+ \rho_8^2 |f_t|^2 + \rho_9^2 |\Delta f|^2 + \rho_{10}^2|\nabla f_t|^2 + \rho_{10}^2|f_{tt}|^2 \right)d(t,x)
        \\ \noalign{\smallskip} \displaystyle
        \hspace{1.0cm} \ +\ \int_0^T \int_\omega\left(\rho_4^2\left|v\right|^2 + \left|(\zeta v)_t\right|^2 +  \left|\zeta\Delta v\right|^2 +  \left|(\widehat{\zeta} v_t)_t\right|^2 + \left|\widehat{\zeta}\Delta v_t\right|^2 + \left|D^4\left( \widehat{\zeta} v \right)\right|^2 \right) dx dt
        \\ \noalign{\smallskip} \displaystyle
        \hspace{1.0cm} \ +\ \|f(0)\|_{\left[H^3(\Omega)\cap H^1_0(\Omega)\right]^N}^2 + \|f_t(0)\|_{H^1_0(\Omega)^N}^2 + \|y(0)\|_{H^5(\Omega)^N}^2,
    \end{array}
\end{equation}
\normalsize
where in (\ref{NormOfY}) we have written $f:= Ly + \nabla p - \chi_\omega v.$ Then, endowing the space $Y$ with $\|\cdot\|_Y$ renders it a Banach space.

Now, we put
\footnotesize
\begin{equation}
    \begin{array}{l}
        F := \Big\{ f \in L^2(Q)^N : \rho_0 f,\ \rho_8 f_t,\ \rho_9 \Delta f,\ \rho_{10}f_{tt} \in L^2(Q)^N,\ \rho_{10}f_t \in L^2\left(0,T; H^1_0(\Omega)^N\right),
        \\ \noalign{\smallskip} \displaystyle
        \hspace{5.0cm} \ f(0) \in \left[H^3(\Omega)\cap H^1_0(\Omega)\right]^N,\ f_t(0) \in H^1_0(\Omega)^N \Big\},
    \end{array}
\end{equation}
\normalsize
\begin{align}
    \begin{split}
        \|f\|_F^2 :=& \int_Q\left(\rho_0^2 |f|^2+ \rho_8^2 |f_t|^2 + \rho_9^2 |\Delta f|^2 + \rho_{10}^2|\nabla f_t|^2 + \rho_{10}^2|f_{tt}|^2 \right)d(t,x) \\
        &\ + \|f(0)\|_{\left[H^3(\Omega)\cap H^1_0(\Omega)\right]^N}^2 + \|f_t(0)\|_{H^1_0(\Omega)^N}^2,
    \end{split}
\end{align}
and also consider the space of initial conditions
$$
G := \left\{ y_0 \in H^5(\Omega)^N \cap V : Ay_0,\,A^2y_0 \in H^1_0(\Omega)^N \right\},
$$
with the same topology as $H^5(\Omega)^N\cap V.$ Then, we define
\begin{equation}
    Z := F \times G,
\end{equation}
The space $Z$ with the natural product topology is also a Banach space.

Finally, we define the mapping $H : Y \rightarrow Z$ by
\begin{equation}
    H(y,p,v) := \left(\frac{Dy}{Dt} - \nabla \cdot \mathcal{T}(y,p) - \chi_\omega v , y(0)\right).
\end{equation}

\subsection{Three lemmas and the conclusion}

\begin{lemma} \label{HIsWellDefined}
The mapping $H : Y \rightarrow Z$ is well-defined, and it is continuous.
\end{lemma}
\begin{proof}
We write $H(y,p,v) = (H_1(y,p,v),H_2(y,p,v)),$ where
\begin{equation} \label{Decomp}
    \begin{cases}
        H_1(y,p,v) := \frac{Dy}{Dt} - \nabla \cdot \mathcal{T}(y,p) - \chi_\omega v; \\
        H_2(y,p,v) := y(0).
    \end{cases}
\end{equation}
There is nothing to prove about $H_2,$ since it is cleary linear and continuous. We will consider only the mapping $H_1$ in what follows.
		
We decompose $H_1(y,v) = h_1(y,p,v) + h_2(y,p,v) + h_3(y,p,v),$ where 
$$
\begin{cases}
h_1(y,p,v):= y_t -\nu(0)\Delta y + \nabla p - \chi_\omega v, \\ h_2(y,p,v):= - \nabla \cdot\left[(\nu(\nabla y)-\nu(0))\nabla y \right], \\
h_3(y,p,v)=(y\cdot \nabla) y.
\end{cases}
$$
By the definition of the norm of $F,$ it follows promptly that 
$$
\|h_1(y,p,v)\|_F < \infty.
$$
Next, we will prove that the quantity $\|h_2(y,p,v)\|_F$ is finite. 


\textbf{CLAIM 1:} $\|\Delta h_2(y,p,v)\|_9 < \infty.$

We notice that
\scriptsize
\begin{equation*}
\begin{array}{l}
|\Delta h_2(y,p,v)| \leqslant C\big[ (r+1)r|r-1||\nabla y|^{r-2}|D^2 y|^3 + (r+1)r|\nabla  y|^{r-1}|D^2 y||D^3 y|+(r+1)|\nabla y|^r|D^4 y|\big]
\\ \noalign{\smallskip} \displaystyle
\hspace{1.5cm}\ = C(D_{1,1} + D_{1,2} + D_{1,3}).
\end{array}    
\end{equation*}
\normalsize

In the case $r=1,$ the term $D_{1,1}$ vanishes and thus $|\Delta h_2|$ is bounded by $C(D_{1,2} + D_{1,3}).$ Otherwise, assuming $r \geqslant 2,$ we have

\begin{equation*}
\begin{array}{l} \displaystyle
\int_Q \rho_9^2 D_{1,1}^2 d(t,x) \leqslant C(r)\int_0^T \rho_9^2 \int_\Omega |\nabla y|^{2(r-2)} |D^2 y|^6 dx\ dt
\\ \noalign{\smallskip} \displaystyle
\hspace{1.5cm} \leqslant  C(r)\int_0^T \rho_{9}^{2}\|D^3 y\|^{2(r-2)}\|D^2 y\|_{L^6(\Omega)}^6 dt
\\ \noalign{\smallskip} \displaystyle
\hspace{1.5cm} \leqslant C(r) \int_0^T \rho_{9}^{2}\|D^3 y\|^{2(r-2)}\|y\|_{H^3(\Omega)}^6 dt
\\ \noalign{\smallskip} \displaystyle
\hspace{1.5cm} \leqslant C(r) \int_0^T \rho_{9}^{2}\|D^3 y\|^{2(r+2)} dt
\\ \noalign{\smallskip} \displaystyle
\hspace{1.5cm} \leqslant C(r)\left( \sup_{\left[0,T\right]}\|\rho_{9}D^3 y\| \right)^{2(r+2)} \int_Q \rho_{9}^{-2(r+1)} d(t,x) < \infty.
\end{array}    
\end{equation*}
In the above equations, we used the continuous immersions: $H^2(\Omega) \hookrightarrow L^\infty(\Omega);$ $H^1(\Omega) \hookrightarrow L^6(\Omega).$ These are valid for $N \leqslant 3,$ see \cite{evans10}, and we will use them tacitly henceforth. 

Now, we obtain the estimate for $D_{1,2}:$ 
\begin{equation*}
\begin{array}{l} \displaystyle
\int_Q \rho_9^2 D_{1,2}^2 d(t,x) \leqslant C(r) \int_0^T \rho_{9}^2 \int_\Omega |\nabla y|^{2(r-1)}|D^2 y|^2|D^3 y|^2 dx\ dt
\\ \noalign{\smallskip} \displaystyle
\hspace{1.5cm} \leqslant C(r) \int_0^T \rho_9^2\|D^3 y\|^{2r}\|D^4 y\|^2 dt
\\ \noalign{\smallskip} \displaystyle
\hspace{1.5cm} \leqslant C(r) \left(\sup_{\left[0,T\right]}\|\rho_{9}D^3 y\| \right)^{2r} \int_Q \rho_{9}^2|D^4 y|^2 d(t,x) < \infty.
\end{array}    
\end{equation*}

Likewise, we show $D_{1,3}$ to be finite, since
\begin{equation*}
\begin{array}{l} \displaystyle
\int_Q \rho_9^2 D_{1,3}^2 d(t,x) \leqslant C(r) \int_0^T \rho_9^2 \int_\Omega |\nabla y|^{2r} |D^4 y|^2dx\ dt
\\ \noalign{\smallskip} \displaystyle
\hspace{1.5cm} \leqslant  C(r) \int_0^T \rho_{9}^{-2r}\left(\rho_{9}\|D^3 y\|\right)^{2r} \left(\rho_{9}\|D^4 y\|\right)^2 dt
\\ \noalign{\smallskip} \displaystyle
\hspace{1.5cm} \leqslant C(r) \sup_{\left[0,T\right]}\|\rho_{9}D^3 y\|^{2r} \int_Q \rho_{9}^2 |D^4 y|^2 d(t,x) < \infty.
\end{array}    
\end{equation*}


\textbf{CLAIM 2:} $\|\partial_t^2 h_2(y,p,v)\|_{10} < \infty.$

We begin with the pointwise estimate,
\begin{equation*}
\begin{array}{l}
|\partial_t^2 h_2(y,p,v)| \leqslant C\big[ (r+1)r|r-1||\nabla y|^{r-2}|\nabla y_t|^2|\Delta y| + (r+1)r|\nabla y|^{r-1}|\nabla y_{tt}||\Delta y| 
\\ \noalign{\smallskip} \displaystyle
\hspace{2.3cm} \ +\ (r+1)r|\nabla y|^{r-1}|\nabla y_t||\Delta y_t| +(r+1)|\nabla y|^r|\Delta y_{tt}|\big]
\\ \noalign{\smallskip} \displaystyle
\hspace{2.05cm} \ = C(D_{2,1}+D_{2,2}+D_{2,3}+D_{2,4}).
\end{array}    
\end{equation*}

As in the previous claim, if $r=1,$ then $D_{2,1}\equiv 0.$ For $r \geqslant 2,$ the next estimate is valid:
\begin{equation*}
\begin{array}{l} \displaystyle
\int_Q \rho_{10}^2 D_{2,1}^2 d(t,x) \leqslant C(r) \int_0^T \rho_{10}^2 \int_\Omega |\nabla y|^{2(r-2)}|\nabla y_{t}|^4|\Delta y|^2 dx\ dt
\\ \noalign{\smallskip} \displaystyle
\hspace{1.5cm} \leqslant C(r) \int_0^T \rho_9^{-2r}  \|\rho_{9} D^3 y\|^{2(r-2)}\|\rho_{9} D^4 y\|^2 \|\rho_{10}  \Delta y_t\|^4 dt
\\ \noalign{\smallskip} \displaystyle
\hspace{1.5cm} \leqslant C(r) \left(\sup_{\left[0,T\right]}\|\rho_{9}D^3 y\| \right)^{2(r-2)} \left(\sup_{\left[0,T\right]}\|\rho_{10}\Delta y_t\| \right)^4 \int_Q \rho_{9}^2|D^4 y|^2 d(t,x)
\\ \noalign{\smallskip} \displaystyle
\hspace{1.5cm} < \infty.
\end{array}    
\end{equation*}

Proceeding similarly, we prove the remaining inequalities:
\begin{equation*}
\begin{array}{l} \displaystyle
\int_Q \rho_{10}^2 D_{2,2}^2 d(t,x) \leqslant C(r) \int_0^T \rho_{10}^2 \int_\Omega |\nabla y|^{2(r-1)}|\nabla y_{tt}|^2|\Delta y|^2 dx\ dt
\\ \noalign{\smallskip} \displaystyle
\hspace{1.5cm} \leqslant C(r) \int_0^T \rho_{10}^{2}\rho_9^{-2r} \rho_{11}^{-2}  \|\rho_9 D^3 y\|^{2(r-1)} \|\rho_{9} D^4 y\|^2 \|\rho_{11} \nabla y_{tt}\|^2 dt
\\ \noalign{\smallskip} \displaystyle
\hspace{1.5cm} \leqslant C(r) \left(\sup_{\left[0,T\right]}\|\rho_{9}D^3 y\| \right)^{2(r-1)} \left(\sup_{\left[0,T\right]}\|\rho_{11}\nabla y_{tt}\| \right)^2 \int_Q \rho_{9}^2|D^4 y|^2 d(t,x)\\ \noalign{\smallskip} \displaystyle
\hspace{1.5cm} < \infty;
\end{array}    
\end{equation*}

\begin{equation*}
\begin{array}{l} \displaystyle
\int_Q \rho_{10}^2 D_{2,3}^2 d(t,x) \leqslant C(r) \int_0^T \rho_{10}^2 \int_\Omega |\nabla y|^{2(r-1)} |\nabla y_{t}|^2|\Delta y_t|^2 dx\ dt
\\ \noalign{\smallskip} \displaystyle
\hspace{1.5cm} \leqslant C(r) \int_0^T \rho_9^{-2(r-1)}\rho_{10}^{-2} \|\rho_9 D^3 y\|^{2(r-1)} \|\rho_{10} D^3 y_t\|^2 \|\rho_{10} \Delta y_{t}\|^2 dt
\\ \noalign{\smallskip} \displaystyle
\hspace{1.5cm} \leqslant C(r) \left(\sup_{\left[0,T\right]}\|\rho_{9}D^3 y\| \right)^{2(r-1)} \left(\sup_{\left[0,T\right]}\|\rho_{10}\Delta y_{t}\| \right)^2 \int_Q \rho_{10}^2|D^3 y_t|^2 d(t,x)\\ \noalign{\smallskip} \displaystyle
\hspace{1.5cm} < \infty;
\end{array}    
\end{equation*}

\begin{equation*}
\begin{array}{l} \displaystyle
\int_Q \rho_{10}^2 D_{2,4}^2 d(t,x) \leqslant C(r) \int_0^T \rho_{10}^2 \int_\Omega |\nabla y|^{2r} |\Delta y_{tt}|^2 dx\ dt
\\ \noalign{\smallskip} \displaystyle
\hspace{1.5cm} \leqslant C(r) \int_0^T \rho_{10}^{2} \rho_9^{-2r} \rho_{11}^{-2} \|\rho_{9} D^3 y\|^{2r} \|\rho_{11}  \Delta y_{tt}\|^2 dt
\\ \noalign{\smallskip} \displaystyle
\hspace{1.5cm} \leqslant C(r) \left(\sup_{\left[0,T\right]}\|\rho_{9}D^3 y\| \right)^{2r} \int_Q \rho_{11}^2|\Delta y_{tt}|^2 d(t,x) \\ \noalign{\smallskip} \displaystyle
\hspace{1.5cm} < \infty.
\end{array}    
\end{equation*}
This finishes the proof of the second claim.


\textbf{CLAIM 3:} $\left\| |\partial_t \nabla h_2(y,p,v)| \right\|_{10} < \infty.$

As before, we begin by considering the pointwise estimate:
\small
\begin{equation*}
    \begin{array}{l}
        |\partial_t \nabla h_2(y,p,v)| \leqslant C\left[(r+1)r|r-1||\nabla y|^{r-2}|\nabla y_t||\Delta y| + (r+1)r|\nabla y|^{r-1}|\Delta y_t| \right. \\ \noalign{\smallskip} \displaystyle
        \hspace{2.9cm} \left.+ (r+1)|\nabla y|^r|D^3 y_t| \right]
        \\ \noalign{\smallskip} \displaystyle
        \hspace{2.2cm} = C(D_{3,1} + D_{3,2} + D_{3,3}). 
    \end{array}
\end{equation*}
\normalsize

Again, if $r=1,$ then we need not consider $D_{3,1},$ since it vanishes. For $r\geqslant 2,$

\begin{equation*}
    \begin{array}{l} \displaystyle
        \int_Q \rho_{10}^2 D_{3,1}^2 d(t,x) \leqslant C(r)\int_0^T \rho_{10}^2 \int_\Omega |\nabla y|^{2(r-2)}|\nabla y_t|^2 |\Delta y|^2 dx\ dt
        \\ \noalign{\smallskip} \displaystyle
        \hspace{1.5cm} \leqslant C(r) \int_0^T \rho_{10}^2 \rho_{9}^{-2r} \|\rho_{9}D^3 y\|^{2(r-2)}\|\rho_{9}D^4 y\|^2\|\rho_{9}\nabla y_t\|^2 dt
        \\ \noalign{\smallskip} \displaystyle
        \hspace{1.5cm} \leqslant C(r) \left(\sup_{\left[0,T\right]}\|\rho_9D^3 y\| \right)^{2(r-2)}\left(\sup_{\left[0,T\right]}\|\rho_9D^4 y\| \right)^2\int_Q \rho_9^2 |\nabla y_t|^2 d(t,x)\\ \noalign{\smallskip} \displaystyle
        \hspace{1.5cm}  < \infty,
    \end{array}
\end{equation*}

\begin{equation*}
    \begin{array}{l} \displaystyle
        \int_Q \rho_{10}^2 D_{3,2}^2 d(t,x) \leqslant C(r)\int_0^T \rho_{10}^2\int_\Omega |\nabla y|^{2(r-1)}|\Delta y_t|^2 dx\ dt
        \\ \noalign{\smallskip} \displaystyle
        \hspace{1.5cm} \leqslant C(r) \int_0^T \rho_9^{-2(r-1)} \|\rho_{9}D^3 y\|^{2(r-1)}\|\rho_{10}\Delta y_t\|^2 dt
        \\ \noalign{\smallskip} \displaystyle
        \hspace{1.5cm} \leqslant C(r)\left(\sup_{\left[0,T\right]}\|\rho_9D^3 y\|^2 \right)^{2(r-1)}\left(\sup_{\left[0,T\right]}\|\rho_{10}\Delta y_t\| \right)^2 < \infty,
    \end{array}
\end{equation*}
and
\begin{equation*}
    \begin{array}{l} \displaystyle
        \int_Q \rho_{10}^2 D_{3,3}^2 d(t,x) \leqslant C(r)\int_0^T\rho_{10}^2 \int_\Omega |\nabla y|^{2r}|D^3 y_t|^2dx\ dt
        \\ \noalign{\smallskip} \displaystyle
        \hspace{1.5cm} \leqslant C(r)\int_0^T \rho_9^{-2r}\|\rho_9D^3 y\|^{2r}\|\rho_{10}D^3 y_t\|^2dt
        \\ \noalign{\smallskip} \displaystyle
        \hspace{1.5cm} \leqslant C(r)\left(\sup_{\left[0,T\right]}\|\rho_{9}D^3 y\| \right)^{2r}\int_Q \rho_{10}^2|D^3 y_t|^2 d(t,x) <\infty.
    \end{array}
\end{equation*}
These inequalities confirm the third claim.


The remaining terms composing the $F-$norm of $h_2(y,p,v),$ $\|h_2(y,p,v)\|_F,$ are norms of lower order derivatives of it, compared to the ones considered above, in adequate weighted $L^2$ spaces. Therefore, these terms are even easier to handle. A similar remark is also true for $\|h_3(y,p,v)\|_F.$ In addition, we can show the continuity of $H$ via estimates which are very similar to the ones that we carried out in the claims above; hence, we omit these computations. This ends the proof of the Lemma.
\end{proof}

\begin{lemma} \label{HIsC1}
The mapping $H$ is strictly differentiable at the origin of $Y,$ with derivative $DH(0,0,0) = \Lambda \in \mathcal{L}(Y,Z)$ given by
\begin{equation} \label{DerivativeAtTheOrigin}
    \Lambda \cdot (y,p,v) = \left( y_t - \nu_0 \Delta y + \nabla p - \chi_\omega v, y(0)\right) = (\Lambda_1\cdot (y,p,v),\Lambda_2\cdot (y,p,v)).
\end{equation}
In fact, $H$ is of class $C^1(Y,Z)$ and, for each $(\overline{y},\overline{p},\overline{v}) \in Y,$ its derivative $DH(\overline{y},\overline{p},\overline{v}) \in \mathcal{L}(Y,Z)$ is given by
\begin{equation} \label{GeneralDH}
    DH(\overline{y},\overline{p},\overline{v})\cdot (y,p,v) = \left(\Lambda_1(\overline{y},\overline{p},\overline{v})\cdot (y,p,v) ,\,\Lambda_2 \cdot (y,p,v) \right),
\end{equation}
where we have written
$$
\begin{cases}
\Lambda_1(\overline{y},\overline{p},\overline{v})\cdot (y,p,v) := \Lambda_1\cdot (y,p,v)  \\
\hspace{3.5cm} - r\nu_1\nabla \cdot \left[  \chi_{\overline{y}} |\nabla \overline{y}|^{r-2}\nabla \overline{y} : \nabla y\, \nabla \overline{y} + |\nabla \overline{y}|^r\nabla y \right] \\
\hspace{3.5cm}+ \left(y \cdot \nabla\right) \overline{y} + \left(\overline{y} \cdot \nabla\right) y ,\\
\nabla \overline{y} : \nabla y := \tr\left( \nabla \overline{y}^\intercal \nabla y \right),\\
\chi_{\overline{y}} \text{ is the indicator function of the set } \{\nabla \overline{y} \neq 0\}.
\end{cases}
$$
\end{lemma}
\begin{proof}
We will only prove the first claim, i.e., that $H$ is strictly differentiable at the origin $(0,0,0) \in Y,$ with $DH(0,0,0)$ being onto $Z.$ There is no additional difficulty to prove the lemma in its full force.
    
We write $H = (H_1,H_2)$ as in (\ref{Decomp}) of Lemma \ref{HIsWellDefined}. Again, it is only necessary to investigate $H_1,$ since $H_2$ is linear and continuous, and therefore $C^\infty.$ Given $(y,p,v),(\overline{y},\overline{p},\overline{v}) \in Y,$ we note that
$$
H_1(y,p,v) - H_2(\overline{y},\overline{p},\overline{v}) - \Lambda_1 \cdot (y-\overline{y}, p - \overline{p},v-\overline{v}) = -\nu_1 D_1 + D_2,
$$
where
$$
D_1 :=  \nabla \cdot\left(|\nabla \overline{y}|^r \nabla \overline{y} - |\nabla y|^r\nabla y \right),
$$
$$
D_2 := (\overline{y} \cdot \nabla)\overline{y} - (y \cdot \nabla) y.
$$
Let us take two positive real numbers, $\epsilon$ and $\delta,$ and we suppose $\|(y,p,v)\|_Y \leqslant \delta,$ $\|(\overline{y},\overline{p},\overline{v})\|_Y \leqslant \delta.$ We must show that we can take $\delta = \delta(\epsilon)$ such that
$$
\|H_1(y,p,v) - H_2(\overline{y},\overline{p},\overline{v}) - \Lambda_1 \cdot (y-\overline{y}, p - \overline{p},v-\overline{v})\|_F \leqslant \epsilon\|(y-\overline{y}, p - \overline{p}, v-\overline{v})\|_Y.
$$
We assume, without loss of generality, that $\delta < 1.$ It is enough to show that
\begin{equation} \label{EnoughToShow}
    \nu_1\|D_1\|_F + \|D_2\|_F \leqslant \epsilon\|(y-\overline{y}, p-\overline{p}, v- \overline{v})\|_Y,
\end{equation}
for a suitable $\delta = \delta(\epsilon).$ To begin with, we observe that
\begin{equation*}
    \begin{array}{l}
        |\Delta D_1| \leqslant C (r+1)r\big[|r-1|\left||\nabla \overline{y}|^{r-2} - |\nabla y|^{r-2} \right||D^2\overline{y}|^3
        \\ \noalign{\smallskip} \displaystyle
        \hspace{1.5cm} +\ |r-1||\nabla y|^{r-2}\left(|D^2\overline{y}|^2 + |D^2 y|^2 \right)|\nabla \overline{y} -\nabla y|
        \\ \noalign{\smallskip} \displaystyle
        \hspace{1.5cm} +\ |r-1|\left(|\nabla \overline{y}|^{r-2} + |\nabla y|^{r-2} \right)|\nabla\overline{y} -\nabla y||D^2\overline{y}||D^3\overline{y}|
        \\ \noalign{\smallskip} \displaystyle
        \hspace{1.5cm} +\ |\nabla y|^{r-1}|D^2\overline{y}-D^2 y||D^3\overline{y}| +|\nabla y|^{r-1}|D^2 y||D^3(\overline{y}-y)|
        \\ \noalign{\smallskip} \displaystyle
        \hspace{1.5cm} +\ |\nabla\overline{y}|^{r-1}|\nabla(\overline{y}-y)||D^4 \overline{y}| + |\nabla y|^r|D^4(\overline{y}-y)| \big]
        \\ \noalign{\smallskip} \displaystyle
        \hspace{1.0cm} = C(r+1)r\left(D_{1,1} + \cdots + D_{1,7} \right).
    \end{array}
\end{equation*}
If $r=1,$ then $D_{1,1} \equiv D_{1,2} \equiv D_{1,3} \equiv 0,$ whereas for $r=2$ we also have $D_{1,1} \equiv 0.$ If $r \geqslant 3,$ we follow estimates similar to the ones we developed in Lemma \ref{HIsWellDefined}, and make use of the immersions we described there, in such a way that
\small
\begin{equation*}
    \begin{array}{l} \displaystyle
        \int_Q \rho_9^2 D_{1,1}^2 d(t,x)
        \\ \noalign{\smallskip} \displaystyle
        \hspace{1.2cm} \leqslant C(r)\int_0^T \rho_9^2\|D^3 (\overline{y}-y)\|^2\left(\|D^3 \overline{y}\|^{2(r-3)} + \|D^3 y\|^{2(r-3)} \right)\|D^2 \overline{y}\|_{L^6(\Omega)}^6 dt
        \\ \noalign{\smallskip} \displaystyle
        \hspace{1.2cm} \leqslant C(r) \int_0^T \rho_9^2\|D^3 (\overline{y}-y)\|^2\left(\|D^3 \overline{y}\|^{2(r-3)} + \|D^3 y\|^{2(r-3)} \right)\|D^3\overline{y}\|^6 dt
        \\ \noalign{\smallskip} \displaystyle
        \hspace{1.2cm} = C(r) \int_0^T \rho_9^{-2r}\|\rho_9D^3 (\overline{y}-y)\|^2\left(\|\rho_9D^3 \overline{y}\|^{2(r-3)} + \|\rho_9D^3 y\|^{2(r-3)} \right)\|\rho_9D^3\overline{y}\|^6 dt
        \\ \noalign{\smallskip} \displaystyle
        \hspace{1.2cm} \leqslant C(r)\delta^{2r}\|(y-\overline{y},p - \overline{p}, v-\overline{v})\|_{Y}^2.
    \end{array}
\end{equation*}
\normalsize
Next, for $r\geqslant 2,$ 
\begin{equation*}
    \begin{array}{l} \displaystyle
        \int_Q \rho_9^2 D_{1,2}^2 d(t,x)
        \\ \noalign{\smallskip} \displaystyle
        \hspace{1.2cm} \leqslant C(r)\int_0^T \rho_9^2\|D^3 y\|^{2(r-2)}\|D^3(\overline{y}-y)\|^2\left(\|D^2\overline{y}\|_{L^4(\Omega)}^4 + \|D^2 y\|_{L^4(\Omega)}^4 \right)dt
        \\ \noalign{\smallskip} \displaystyle
        \hspace{1.2cm} \leqslant C(r)\int_0^T\rho_9^2\|D^3 y\|^{2(r-2)}\|D^3(\overline{y}-y)\|^2\left(\|D^3 \overline{y}\|^4 + \|D^3 y\|^4 \right)dt
        \\ \noalign{\smallskip} \displaystyle
        \hspace{1.2cm} \leqslant C(r)\delta^{2r}\|(y-\overline{y}, p-\overline{p}, v-\overline{v})\|_Y^2,
    \end{array}
\end{equation*}
\small
\begin{equation*}
    \begin{array}{l} \displaystyle
        \int_Q \rho_9^2D_{1,3}^2d(t,x)
        \\ \noalign{\smallskip} \displaystyle
        \hspace{1.0cm} \leqslant C(r)\int_0^T\rho_9^2\left(\|D^3 \overline{y}\|^{2(r-2)} + \|D^3 y\|^{2(r-2)} \right)\|D^3(\overline{y}-y)\|^2\|D^4 \overline{y}\|^2\|D^3 \overline{y}\|^2 dt
        \\ \noalign{\smallskip} \displaystyle
        \hspace{1.0cm} \leqslant C(r)\delta^{2r}\|(y-\overline{y},p-\overline{p},v-\overline{v})\|_Y^2.
    \end{array}
\end{equation*}
\normalsize
Now, for every $r \geqslant 1,$
\begin{equation*}
    \begin{array}{l} \displaystyle
        \int_Q \rho_9^2D_{1,4}^2d(t,x) \leqslant C(r)\int_0^T \rho_9^2\|D^3 y\|^{2(r-1)}\|D^4(\overline{y}-y)\|^2\|D^3\overline{y}\|^2 dt
        \\ \noalign{\smallskip} \displaystyle
        \hspace{1.2cm} \leqslant C(r)\delta^{2r}\|(y-\overline{y},p-\overline{p},v-\overline{v})\|_Y^2,
    \end{array}
\end{equation*}
\begin{equation*}
    \begin{array}{l} \displaystyle
        \int_Q \rho_9^2D_{1,5}^2d(t,x) \leqslant C(r)\int_0^T \rho_9^2 \|D^3 y\|^{2(r-1)}\|D^4 y\|^2\|D^3 (\overline{y}-y)\|^2 dt
        \\ \noalign{\smallskip} \displaystyle
        \hspace{1.2cm} \leqslant C(r) \delta^{2r}\|(y-\overline{y},p-\overline{p}, v-\overline{v})\|_Y^2,
    \end{array}
\end{equation*}
\begin{equation*}
    \begin{array}{l} \displaystyle
        \int_Q \rho_9^2D_{1,6}^2d(t,x) \leqslant C(r)\int_0^T \rho_9^2\|D^3 \overline{y}\|^{2(r-1)}\|D^3(\overline{y}-y)\|^2\|D^4\overline{y}\|^2 dt
        \\ \noalign{\smallskip} \displaystyle
        \hspace{1.2cm} \leqslant C(r)\delta^{2r}\|(y-\overline{y},p-\overline{p}, v-\overline{v})\|_Y^2,
    \end{array}
\end{equation*}
\begin{equation*}
    \begin{array}{l} \displaystyle
        \int_Q \rho_9^2D_{1,7}^2d(t,x) \leqslant C(r)\int_0^T \rho_9^2\|D^3 y\|^{2r}\|D^4(\overline{y}-y)\|^2 dt
        \\ \noalign{\smallskip} \displaystyle
        \hspace{1.2cm} \leqslant C(r)\delta^{2r}\|(y-\overline{y},p-\overline{p}, v-\overline{v})\|_Y^2.
    \end{array}
\end{equation*}
Summing up, the computations we carried out above yield
$$
\|\Delta D_1\|_9 \leqslant C(r)\delta^{r}\|(y-\overline{y},p-\overline{p},v-\overline{v})\|_Y.
$$
We can treat the remaining terms composing the $F-$norm of $D_1$ likewise, as we argued in Lemma \ref{HIsWellDefined}. Dealing with $D_2$ is even simpler, since it involves lower order derivatives of $y.$ In this way, we deduce that
$$
\nu_1\|D_1\|_F + \|D_2\|_F \leqslant C(r)\delta\|(y-\overline{y},p-\overline{p},v-\overline{v})\|_Y.
$$
Thus, it suffices to take any positive $\delta < \min(1,\epsilon/C(r))$ in order to finish the proof.
\end{proof}    
    
\begin{lemma} \label{BijectiveDHAtOrigin}
The linear operator $DH(0,0,0) : Y \rightarrow Z$ is continuous and onto. Furthermore, there exists a constant $M>0$ such that
\begin{equation} \label{Injective}
    \|(y,p,v)\|_Y \leqslant M\|DH(0,0,0)\cdot (y,p,v)\|_Z
\end{equation}
\end{lemma}
\begin{proof}
The continuity of $DH(0,0,0)$ follows promptly from the definition of the norms of $Y$ and $Z.$ As for the surjectiveness of this mapping, let us consider $(f,y_0) \in Z.$ We take $(y,p,v)$ as the state-pressure-control tuple given by Theorem \ref{ControlOfLinearSystem}. By the estimates we proved in subsection \ref{Additional}, namely (\ref{estimate0}), (\ref{estimate1}), (\ref{estimate2}), (\ref{estimate3}), (\ref{estimate4}), (\ref{estimate5}), and (\ref{estimate6}), together with Lemma \ref{RegularControl}, the membership $(y,p,v) \in Y$ is valid. Moreover,
$$
DH(0,0,0)\cdot (y,p,v) = (y_t - \nu_0\Delta y +\nabla p - \chi_\omega v, y(0)) = (f,y_0),
$$
where the last equality holds by the choice of $(y,p,v);$ hence, $DH(0,0,0)$ is onto $Z.$ By the aforementioned estimates, (\ref{Injective}) follows easily. This establishes the lemma.       
\end{proof}

\subsection{Proof of Theorem \ref{MainThm}}

According to Lemmas \ref{HIsWellDefined}, \ref{HIsC1} and \ref{BijectiveDHAtOrigin}, it is licit to apply Theorem \ref{Liusternik}. This result allows us to deduce the existence of $\eta > 0$ such that, for each $(f,y_0) \in Z$ subject to 
\begin{equation} \label{SmallData}
    \|(f,y_0)\|_Z < \eta,
\end{equation}
the equation
\begin{equation} \label{Soluble}
    H(y,p,v) = (f,y_0)
\end{equation}
has a solution $(y,p,v) \in Y$ which satisfies 
\begin{equation} \label{BddnssOfSoln}
    \|(y,p,v)\|_Y < B \eta,
\end{equation}
for a suitable constant $B > 0$ which is independent of $\eta.$ Explicitly, we can take $B := (M^{-1} - \delta)^{-1},$ where $M>0$ is given by Lemma \ref{BijectiveDHAtOrigin} (cf. (\ref{Injective})), and where we select the positive constant $\delta < M^{-1}$ such that $H$ satisfies condition \textit{(iii)} of Theorem \ref{Liusternik}. Such a constant $\delta$ does in fact exist by Lemma \ref{HIsC1}. 

In particular, taking $f\equiv 0,$ inequality (\ref{SmallData}) reads
\begin{equation}
    \|y_0\|_{H^5(\Omega)^N} < \eta.
\end{equation}
Since $(y,p,v) \in Y,$ we have \eqref{eq:ControlCondn}, and alonside \eqref{Soluble}, we see that $(y,p,v)$ does solve \eqref{Model}.

%% file: Numerical.tex
\section{Numerical analysis} \label{sec4}

\subsection{Proof of the convergence of the algorithm}

The proof of this result is straightforward once we have established Lemmas \ref{HIsC1} and \ref{BijectiveDHAtOrigin}. We present it here for completeness.

Firstly, we observe that Lemma \ref{BijectiveDHAtOrigin} ensures that $(y^{n+1},p^{n+1},v^{n+1})$ is well-defined in terms of $(y^n,p^n,v^n),$ since in this lemma we showed that $DH(0,0,0)$ is bijective. Furthermore, we have $\|DH(0,0,0)^{-1}\|_{\mathcal{L}(Z,Y)} \leqslant M,$ according to the notations of this lemma. 

Next, we take $y_0 \in G,$ with $\|y_0\|_{H^5(\Omega)^N} < \eta,$ and we let $(y,p,v) \in Y$ be the solution of $H(y,p,v) = (0,y_0).$ We also consider $0<\epsilon < (2M)^{-1}.$ By Lemma \ref{HIsC1}, there exists $\delta >0$ such that the relations
$$
(y,p,v)\in Y \text{ and } (\overline{y},\overline{p},\overline{v}) \in Y,\ \|(y-\overline{y}, p-\overline{p}, v-\overline{v})\|_Y \leqslant \delta
$$
imply
$$
\|DH(y,p,v) - DH(\overline{y},\overline{p},\overline{v})\|_{\mathcal{L}(Y,Z)} \leqslant \epsilon.
$$
Shrinking $\eta,$ if necessary, we can assume $\eta \leqslant \delta.$ Employing Lemma \ref{HIsC1} once more, we find $\kappa = \kappa(y,p,v) \in \left]0,1\right[$ such that $(\overline{y},\overline{p},\overline{v}) \in Y$ and $\|(y-\overline{y},p-\overline{p},v-\overline{v})\|_Y \leqslant \kappa$ together imply 
$$
\|H(\overline{y},\overline{p},\overline{v}) - H(y,p,v) - DH(y,p,v)\cdot (\overline{y}-y,\overline{p}-p,\overline{v}-v)\|_Z \leqslant \epsilon\|(\overline{y}-y,\overline{p}-p,\overline{v}-v)\|_Y.
$$
We write $e^n := (y^n, p^n, v^n) - (y,p,v),$ and let us assume $\|e^0\|_Y \leqslant \kappa.$ By the algorithm,
\begin{align*}
    e^{n+1} =& - DH(0,0,0)^{-1}\left[H(y^n,p^n,v^n)-H(y,p,v) - DH(y,p,v)               \cdot e^n \right] \\
             & - DH(0,0,0)^{-1}\left[DH(y,p,v) - DH(0,0,0) \right]\cdot e^n,
\end{align*}
whence
\begin{align} \label{Numeric1}
  \begin{split}    
    \|e^{n+1}\|_Y \leqslant M&\left\{\|H(y^n,p^n,v^n)-H(y,p,v) - DH(y,p,v)e^n\| \right. \\
    &\hspace{0.25cm}\left.+ \|\left[DH(y,p,v) - DH(0,0,0) \right]\cdot e^n\|\right\}.
  \end{split}
\end{align}
Assuming inductively that $\|e^n\|_Y \leqslant \kappa,$ which holds true for $n=0,$ it follows that
\begin{equation} \label{Numeric2}
    \|e^{n+1}\|_Y \leqslant 2M\epsilon\|e^n\|_Y.
\end{equation}
Thus, we also have $\|e^{n+1}\|_Y \leqslant \kappa.$ By induction, it follows that $\|e^n\|_Y\leqslant \kappa,$ for every $n;$ hence, it is always possible to pass from (\ref{Numeric1}) to (\ref{Numeric2}). Let us take $\theta := 2M\epsilon.$ Applying inequality (\ref{Numeric2}) iteratively in $n,$ we conclude that
$$
\|e^n\|_Y \leqslant \theta^n\|e_0\|_Y.
$$
This proves Theorem \ref{NumThm}.

\subsection{Implementation of the algorithm}

To implement the fixed-point numerical algorithm, we proceed in two steps. Firstly, it is necessary to implement a solver for the control problem of the forced Stokes system. We begin with the variational problem (\ref{VariationalProblem}) and adequately reformulate it to achieve a mixed formulation, as in \cite{fernandez2017numerical,fernandez2015theoretical}. Below, we recall the main ideas for $N=2.$ After treating the linear problem, we iterate it by updating the source term according to our algorithm.

Under the notations of the proof of Theorem \ref{ControlOfLinearSystem} (see (\ref{PontryaginMinPrinciple})), we define $u := \rho_3^{-1}\left(L^*\varphi + \nabla \pi\right),$ $m := \rho_4^{-1}\varphi,$ and $k := \rho_4^{-1}\pi.$ Let us introduce the spaces
\scriptsize
$$
Z := \left\{ (m^\prime,k^\prime) : m^\prime \in L^2(0,T; H^1_0(\Omega)^2,\ m^\prime_t \in L^2(Q)^2,\ k^\prime \in L^2(0,T; H^1(\Omega)),\ \int_\Omega k^\prime\ dx = 0 \text{ a.e. } \right\},
$$
\normalsize
and
$$
W:= L^2(Q)^2 \times Z,\ M := L^2(0,T;H^1_0(\Omega)^2)\times L^2(Q),
$$
as well as the bilinear forms $b_1 : W \times W \rightarrow \mathbb{R},$ $B,B_1 : W \times M \rightarrow \mathbb{R}$ by
$$
b_1((u,m,k),(u^\prime, m^\prime, k^\prime)) := \int_Q \left\{u \cdot u^\prime + \chi m \cdot m^\prime \right\}d(t,x),
$$
\small
$$
B((u,m,k),(\lambda,\mu)) := \int_Q \left\{ \lambda \cdot \left[u+\rho_3^{-1}\left(\rho_4 m\right)_t + \nabla\left(\rho_4 k \right) \right] - \nabla\left( \rho_3^{-1}\lambda \right): \nabla\left(\rho_4 m \right) \right\} d(t,x)
$$
\normalsize
and
$$
B_1((u,m,k),(\lambda,\mu)) = B((u,m,k),(\lambda,\mu)) - \int_Q \rho_3^{-1} \mu \nabla \cdot\left(\rho_4 m \right) d(t,x).
$$
The last element we introduce is the linear form $\widetilde{\Lambda} : W \rightarrow \mathbb{R},$ which is given by
$$
\langle \widetilde{\Lambda}, (u,m,k) \rangle := \int_Q \rho_4 f m\ d(t,x) + \int_\Omega (\rho_4 m)(0)y_0\ dx.
$$
We reformulate problem (\ref{VariationalProblem}) as: find $(u,m,k) \in W$ and multipliers $(\lambda,\mu) \in M$ such that
\footnotesize
$$
\begin{cases}
b_1((u,m,k),(u^\prime,m^\prime,k^\prime)) + B_1((u^\prime,m^\prime,k^\prime),(\lambda,\mu)) = \langle \widetilde{\Lambda}, (u^\prime,m^\prime,k^\prime) \rangle, \text{ for all } (u^\prime,m^\prime,k^\prime) \in W, \\
B_1((u,m,k),(\lambda^\prime,\mu^\prime)) = 0, \text{ for all } (\lambda^\prime,\mu^\prime) \in M.
\end{cases}
$$
\normalsize
After we solve it, we recover the control and corresponding velocity field of the linear control problem (\ref{VariationalProblem}) via 
$$
v = - \chi \rho_4^{-1} m \text{ and } y = \rho_{3}^{-1} u.
$$
If we assume that $\Omega$ is polygonal, it is simple to find finite dimensional approximations $W_h$ and $M_h$ of the spaces $W$ and $M.$

\subsection{A numerical experiment}

In the sequel, we will employ the FreeFem++ library of C++; see \url{http://www.freefem.org/ff++} for more informations. In Table \ref{tab:Params4Sim}, we describe the datum we used to apply the quasi-Newton method for (\ref{Model}).

\begin{table}[!htp]
\centering
\begin{tabular}{@{}cccccccc@{}}
\toprule
$N$ & $\Omega$             & $\omega$                 & $T$   & $y_0$                  & $\nu_0$ & $\nu_1$ & $r$ \\ \midrule
$2$ & $\left]0,3\right[^2$ & $\left]0.5,2.5\right[^2$ & $1$ & $\nabla \times \psi_0$ & $100$   & $0.01$  & $2$ \\ \bottomrule
\end{tabular}
\caption{Data we used in the simulations. Above, we denote by $\psi_0$ the initial streamline function, which we fix as $\psi_0(x_1,x_2) = (x_1x_2)^2(1-x_1)^2(1-x_2)^2.$}
\label{tab:Params4Sim}
\end{table}
We illustrate in Figure \ref{fig:Meshes} the $2D$ mesh of $\Omega,$ and the $3D$ mesh of the cylinder $Q.$ In Figure \ref{fig:initCondn}, we show both components of the initial state $y(0) = y_0.$ 

\begin{figure}[!ht]
    \centering
    \subfigure[Horizontal ($2D$) mesh]{\includegraphics[width=40mm, scale=0.4]{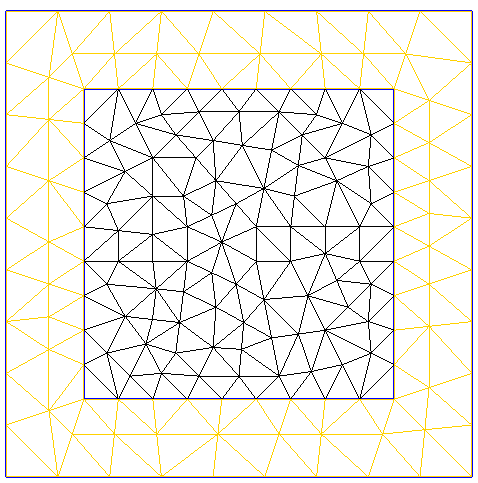}}
    \subfigure[Three-dimensional mesh]{\includegraphics[width=60mm, scale=0.5]{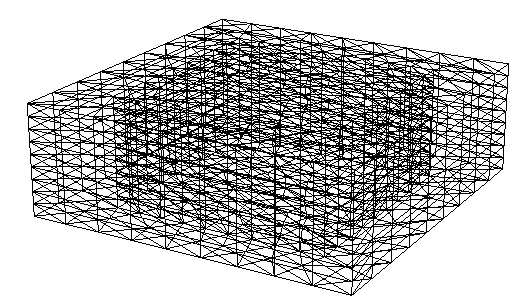}}
    \caption{Meshes. The volume is nine, there are $8748$ tetrahedra, $1810$ vertices, and $1944$ boundary triangles.}
    \label{fig:Meshes}
\end{figure}
\FloatBarrier

\begin{figure}[!ht]
    \centering
    \subfigure[$y_1(0)$]{\includegraphics[width=60mm, scale=0.5]{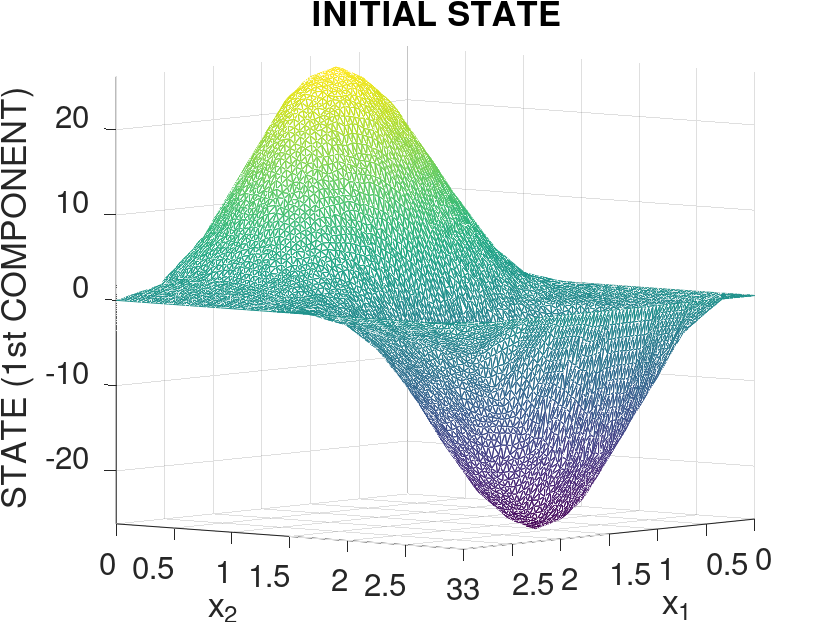}}
    \subfigure[$y_2(0)$]{\includegraphics[width=60mm, scale=0.5]{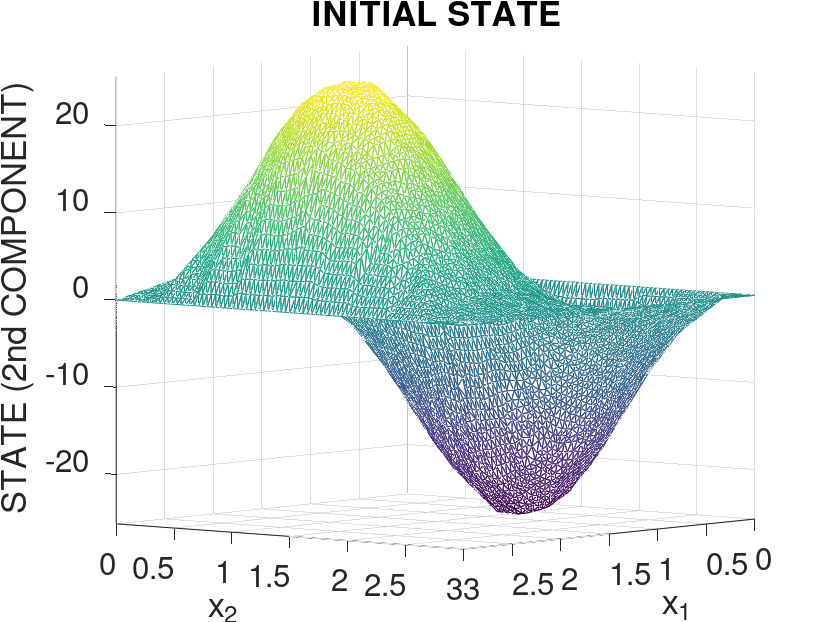}}
    \caption{Initial condition $y_0.$}
    \label{fig:initCondn}
\end{figure}
\FloatBarrier

Our stopping criterion is
$$
\frac{\|y^{n+1}-y^n\|_{L^2(Q)}}{\|y^n\|_{L^2(Q)}} \leqslant \epsilon,
$$
with $\epsilon = 10^{-8}.$ We took as the initial guess $(y^0,p^0,v^0) = (0,0,0).$ We attained convergence after six iterates, with a rate of $4.68$. 

We begin to illustrate the overall behavior of the control and of the state we computed through the plots of some cross-sections in space of them. On the one hand, for the control, we plot the $x_1 = 0.9$ and $x_1 = 2.1$ cuts in Figures \ref{fig:FirstCrossSecControl} and \ref{fig:SecondCrossSecControl}, respectively. On the other hand, we provide the surfaces comprising the values of the state components, relative to these cuts, in Figures \ref{fig:FirstCrossSecState} and \ref{fig:SecondCrossSecState}. 

\begin{figure}[!ht]
    \centering
    \subfigure[$v_1(\cdot,0.9,\cdot)$]{\includegraphics[width=60mm, scale=0.5]{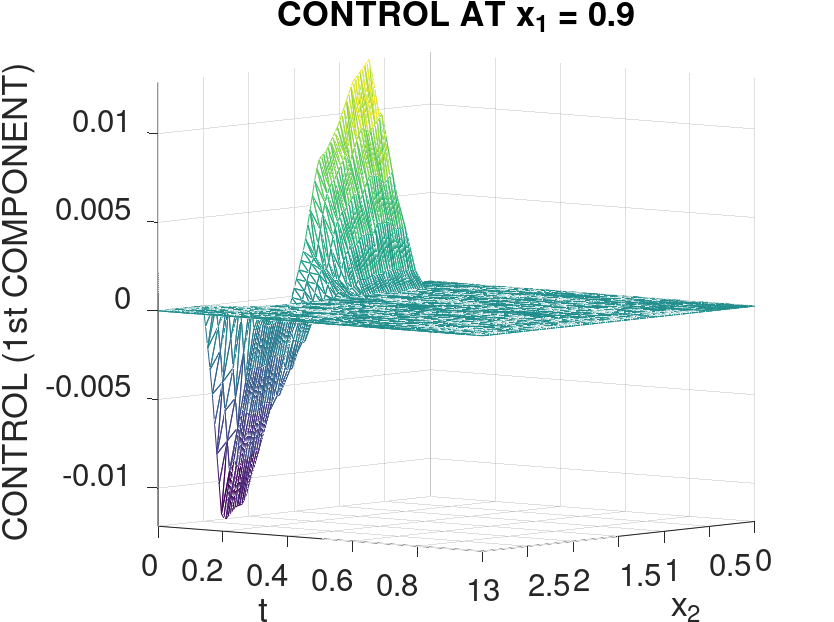}}
    \subfigure[$v_2(\cdot,0.9,\cdot)$]{\includegraphics[width=60mm, scale=0.5]{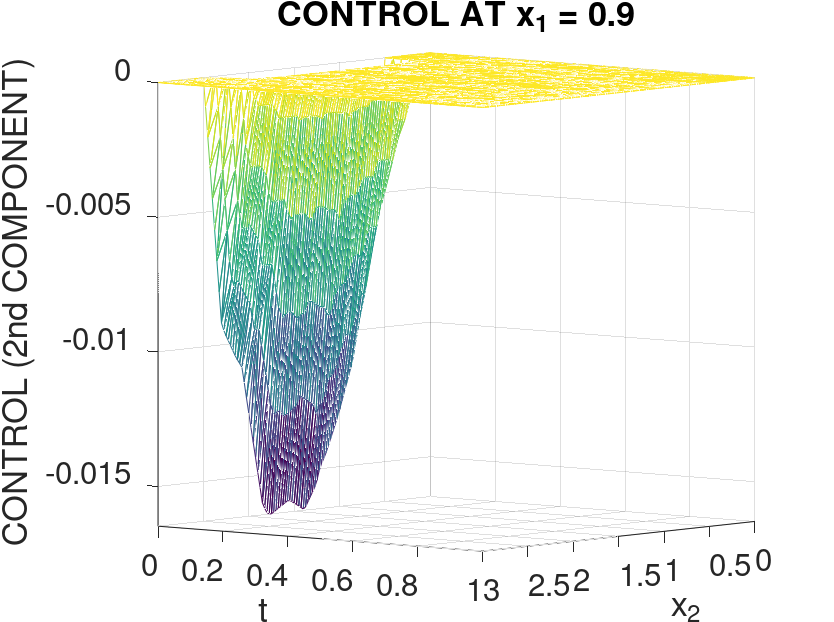}}
    \caption{Cross sections $x_1 = 0.9$ of the control components.}
    \label{fig:FirstCrossSecControl}
\end{figure}
\FloatBarrier

\begin{figure}[!ht]
    \centering
    \subfigure[$v_1(\cdot,2.1,\cdot)$]{\includegraphics[width=60mm, scale=0.5]{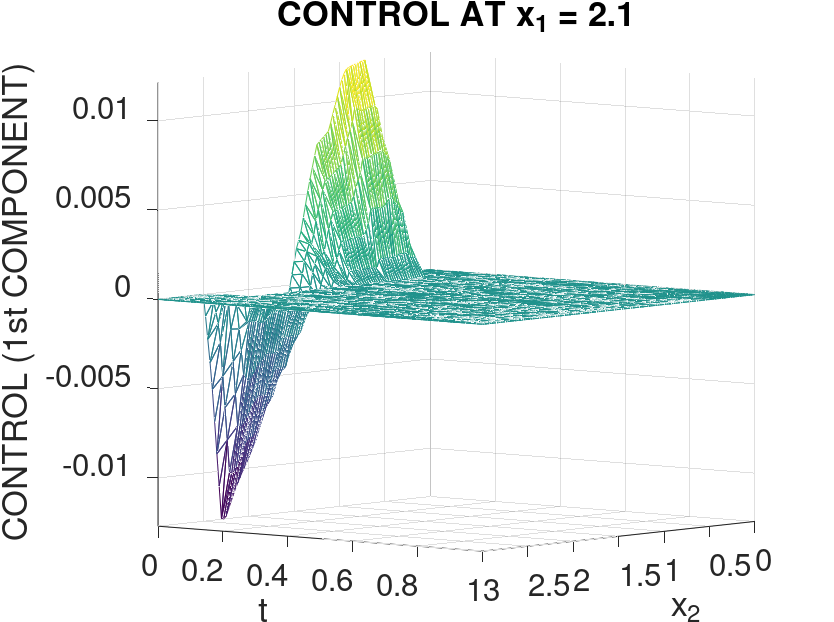}}
    \subfigure[$v_2(\cdot,2.1,\cdot)$]{\includegraphics[width=60mm, scale=0.5]{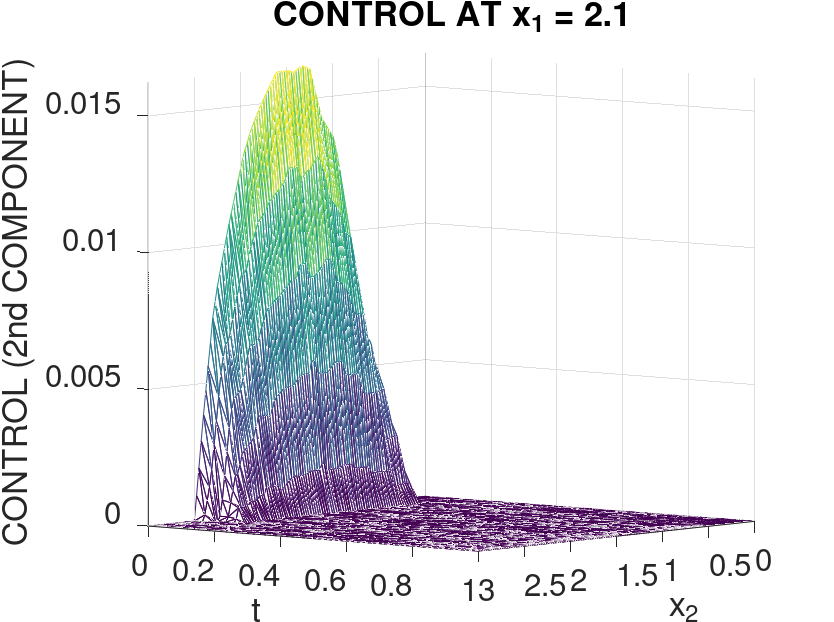}}
    \caption{Cross sections $x_1 = 2.1$ of the control components.}
    \label{fig:SecondCrossSecControl}
\end{figure}
\FloatBarrier

\begin{figure}[!ht]
    \centering
    \subfigure[$y_1(\cdot,0.9,\cdot)$]{\includegraphics[width=60mm, scale=0.5]{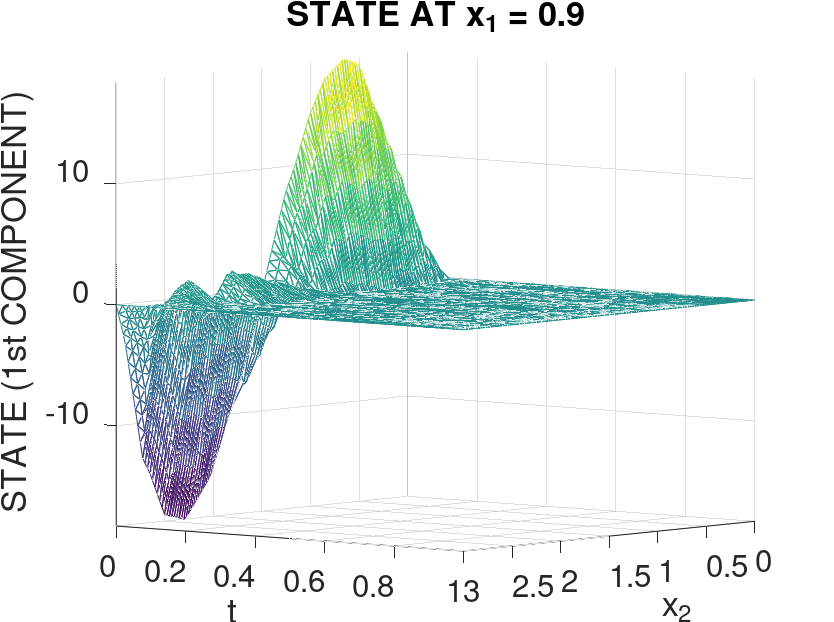}}
    \subfigure[$y_2(\cdot,0.9,\cdot)$]{\includegraphics[width=60mm, scale=0.5]{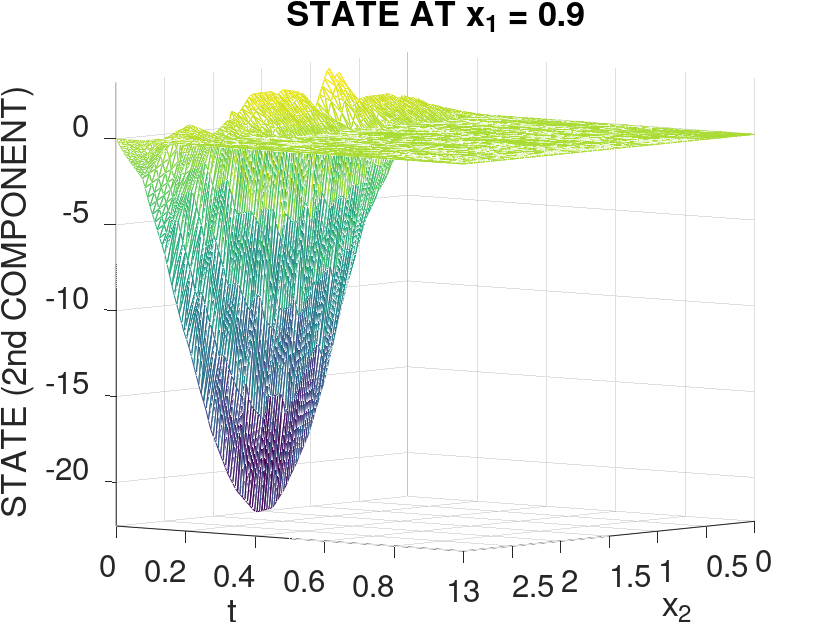}}
    \caption{Cross sections $x_1 = 0.9$ of the state components.}
    \label{fig:FirstCrossSecState}
\end{figure}
\FloatBarrier

\begin{figure}[!ht]
    \centering
    \subfigure[$y_1(\cdot,0.8,\cdot)$]{\includegraphics[width=60mm, scale=0.5]{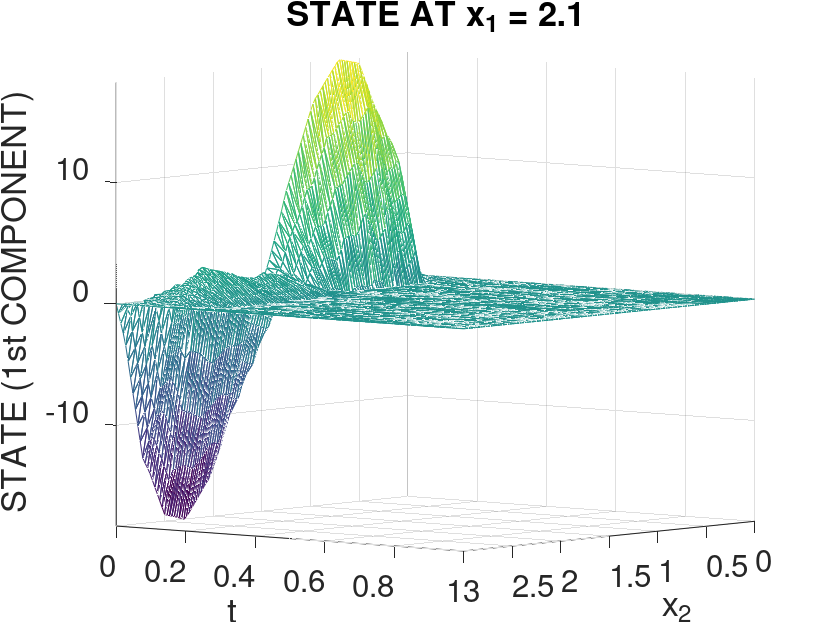}}
    \subfigure[$y_2(\cdot,0.8,\cdot)$]{\includegraphics[width=60mm, scale=0.5]{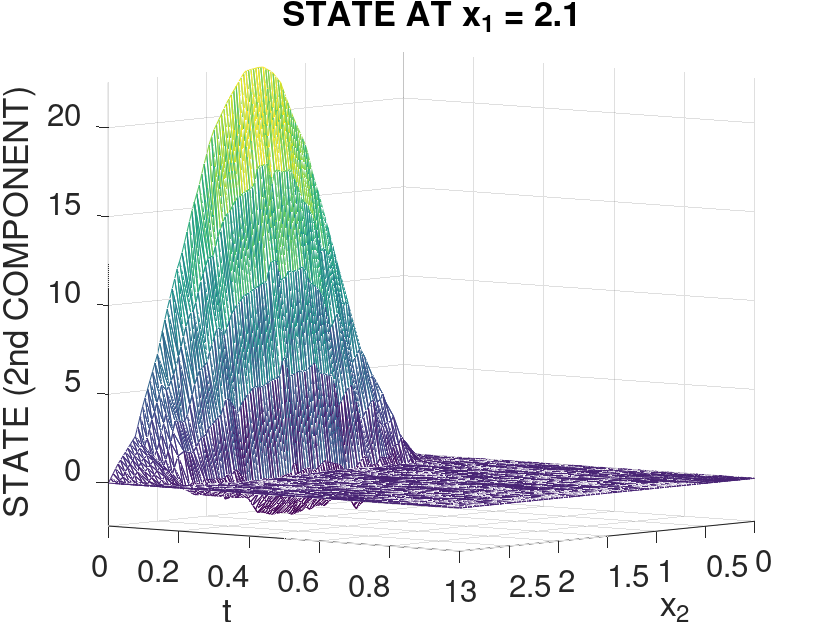}}
    \caption{Cross sections $x_1 = 2.1$ of the state components.}
    \label{fig:SecondCrossSecState}
\end{figure}
\FloatBarrier

The time evolution of the norms of the control and of the corresponding state is what we illustrate in Figure \ref{fig:timeEvoNorms}. It corroborates our theoretical findings as these norms decay exponentially. To further illustrate the control, we provide a surface of its values at initial time in Figure \ref{fig:initControl}. Then, we give some insight into the dynamics of the problem by showcasing some heat maps of the control and of its corresponding state. Namely, in Figure \ref{fig:ControlHMs}, we illustrate the control at time $t=0.15$ --- it is already considerably small, as we would expect from Figure \ref{fig:timeEvoNorms}. For several times, viz., for each $t\in \left\{0.15,\,0.25,\,0.35,\, 0.45\right\},$ we give a heat map of the first (respectively, second) component of the velocity field in Figure \ref{fig:State1HMs} (respectively, Figure \ref{fig:State2HMs}).

\begin{figure}[!ht]
    \centering
    \subfigure{\includegraphics[width=60mm, scale=0.5]{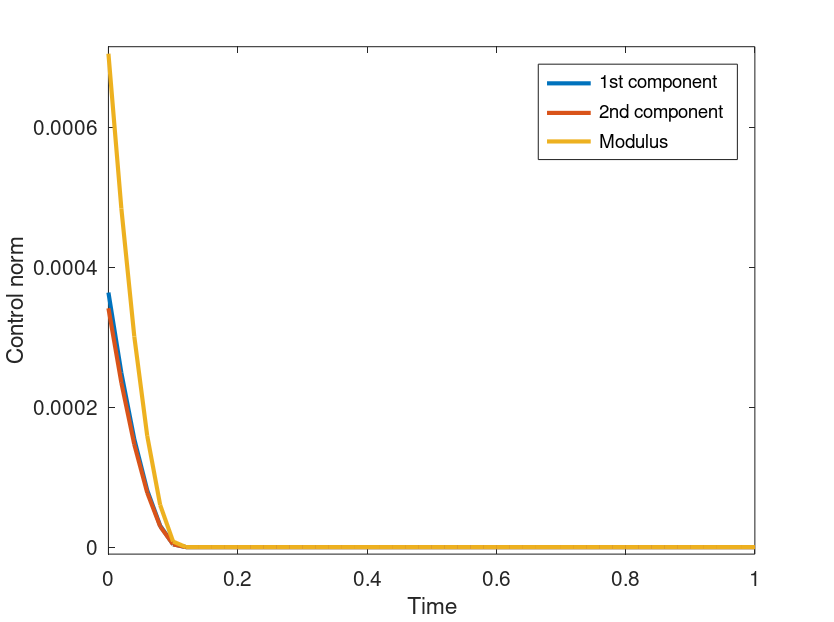}}
    \subfigure{\includegraphics[width=60mm, scale=0.5]{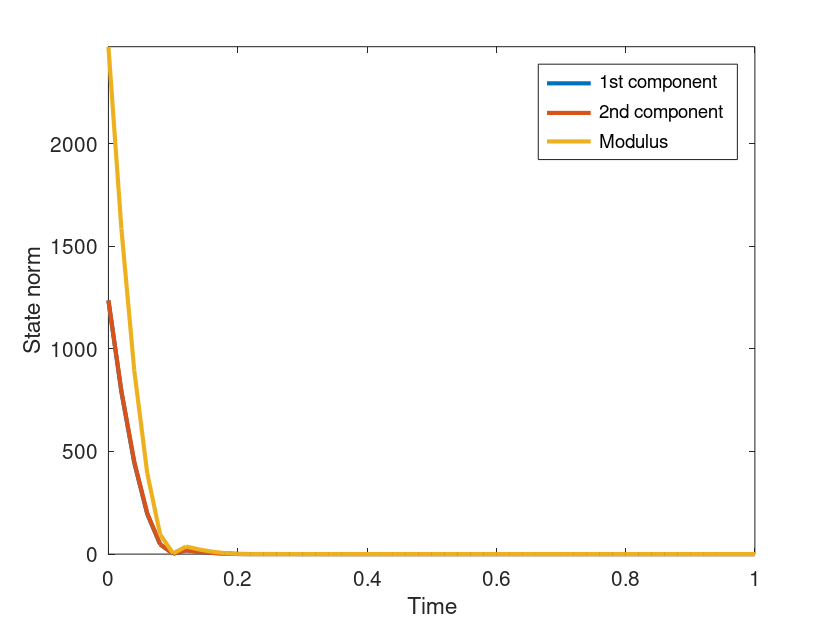}}
    \caption{Time evolution of the control (left panel) and state (right panel) norms. We referred to the sum of the norms of the components as ``Modulus''.}
    \label{fig:timeEvoNorms}
\end{figure}
\FloatBarrier

\begin{figure}[!ht]
    \centering
    \subfigure[$v_1(0)$]{\includegraphics[width=60mm, scale=0.5]{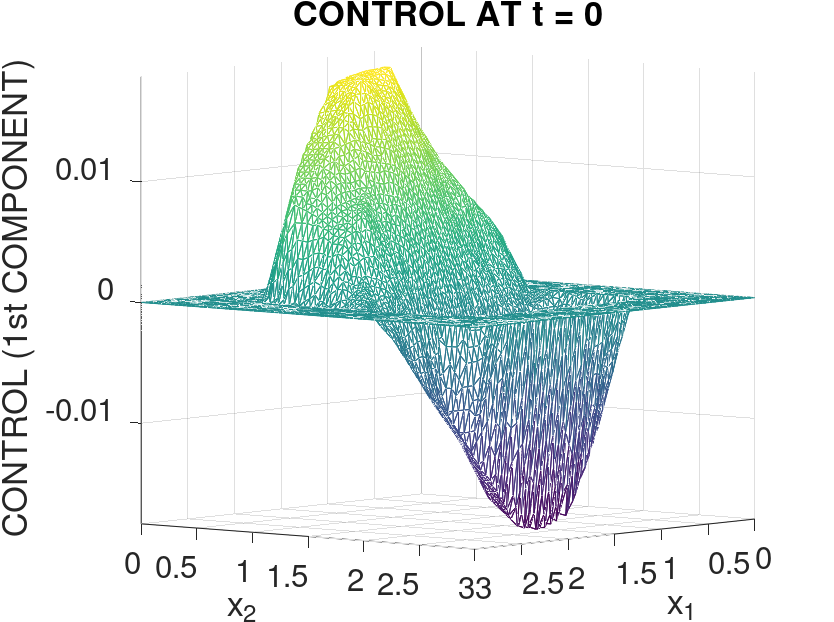}}
    \subfigure[$v_2(0)$]{\includegraphics[width=60mm, scale=0.5]{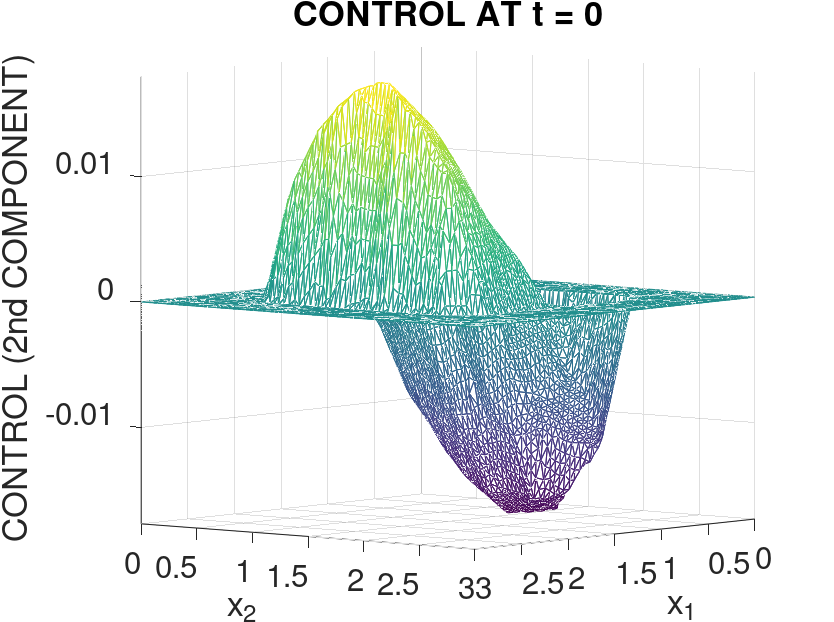}}
    \caption{The control components at time zero.}
    \label{fig:initControl}
\end{figure}
\FloatBarrier

\begin{figure}[!ht]
    \centering
    \subfigure[$v_1(0.15)$]{\includegraphics[scale=0.4]{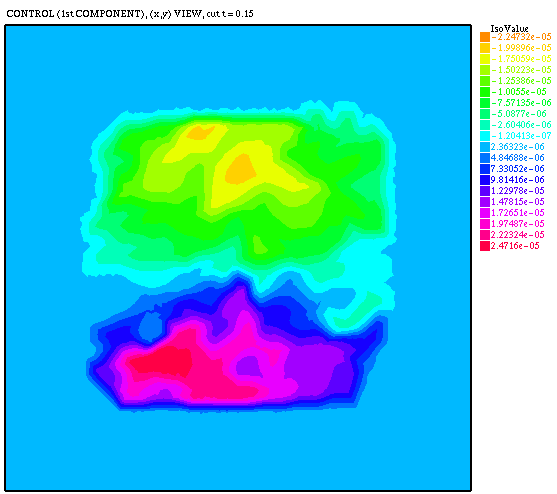}}
    \subfigure[$v_2(0.15)$]{\includegraphics[scale=0.4]{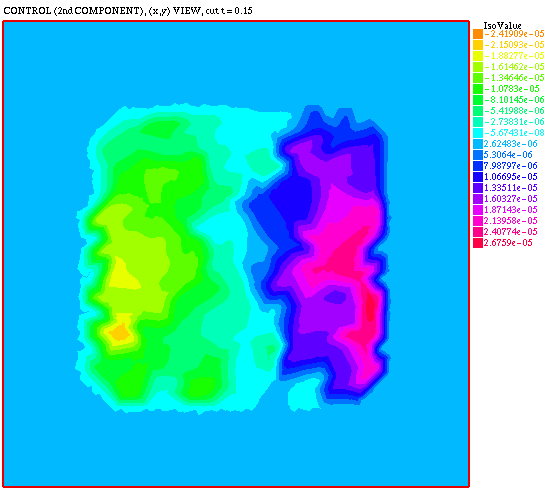}}
    \caption{Heat maps of the control components at $t= 0.15.$}
    \label{fig:ControlHMs}
\end{figure}
\FloatBarrier

\begin{figure}[!ht]
    \centering
    \subfigure[$y_1(0.15)$]{\includegraphics[scale=0.4]{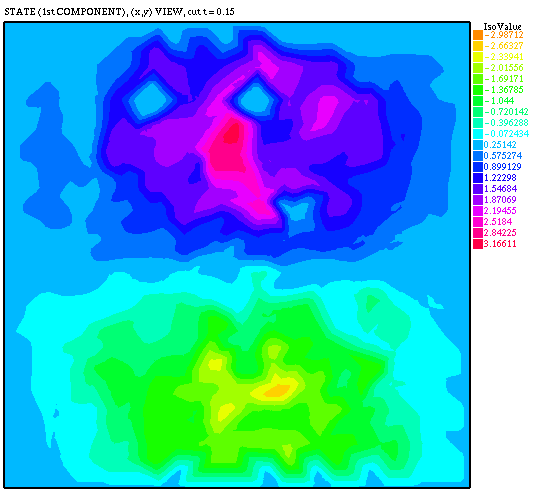}}
    \subfigure[$y_1(0.25)$]{\includegraphics[scale=0.4]{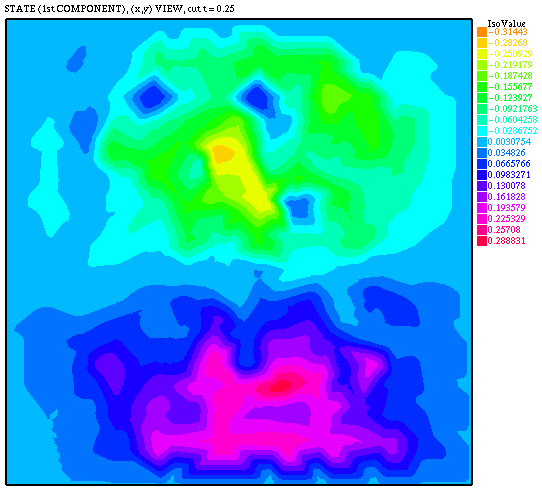}}
    \subfigure[$y_1(0.35)$]{\includegraphics[scale=0.4]{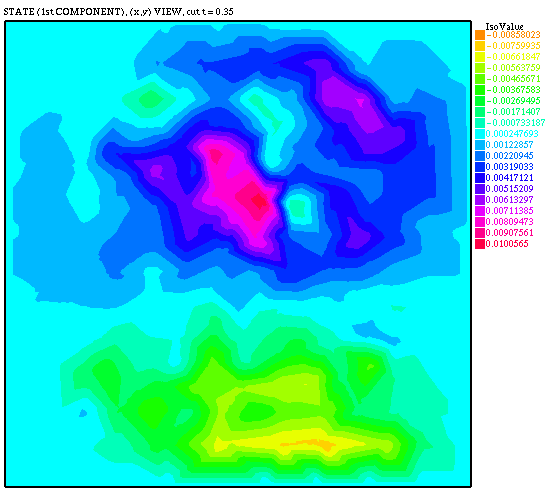}}
    \subfigure[$y_1(0.45)$]{\includegraphics[scale=0.4]{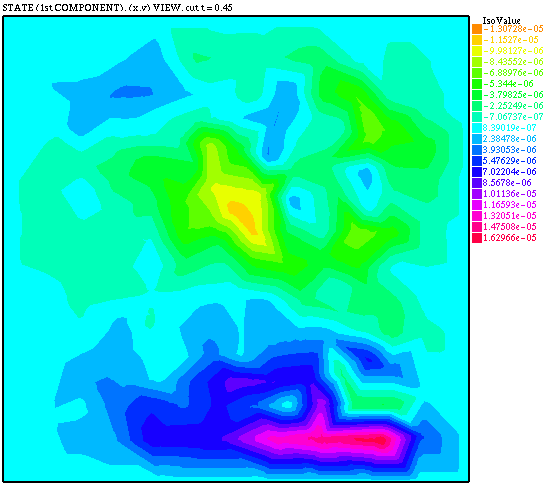}}
    \caption{Heat maps of the first state component at diverse times.}
    \label{fig:State1HMs}
\end{figure}
\FloatBarrier

\begin{figure}[!ht]
    \centering
    \subfigure[$y_2(0.15)$]{\includegraphics[scale=0.4]{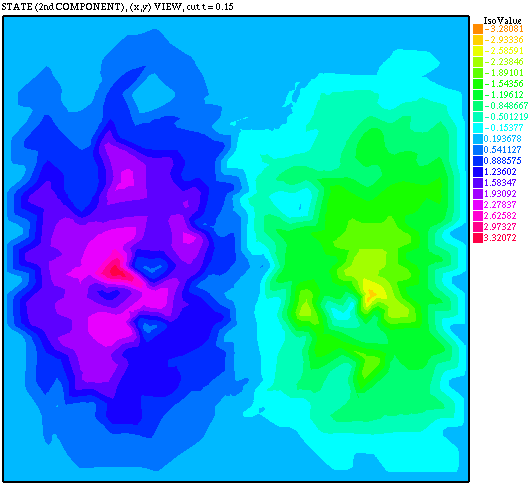}}
    \subfigure[$y_2(0.25)$]{\includegraphics[scale=0.4]{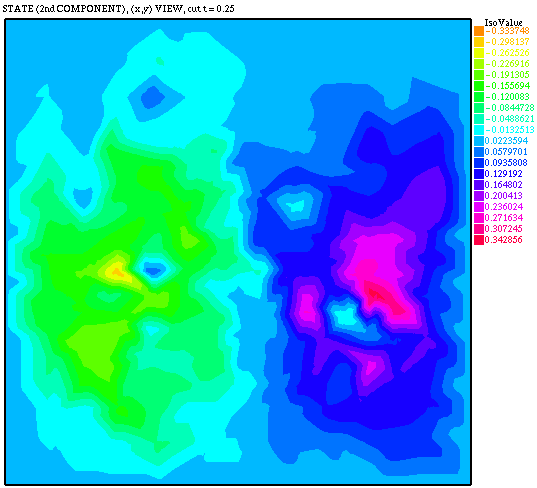}}
    \subfigure[$y_2(0.35)$]{\includegraphics[scale=0.4]{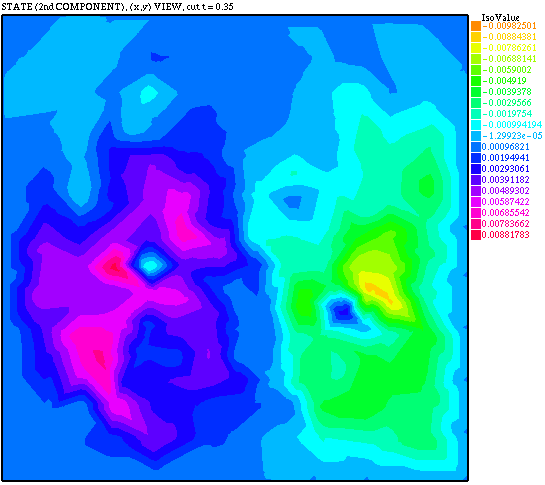}}
    \subfigure[$y_2(0.45)$]{\includegraphics[scale=0.4]{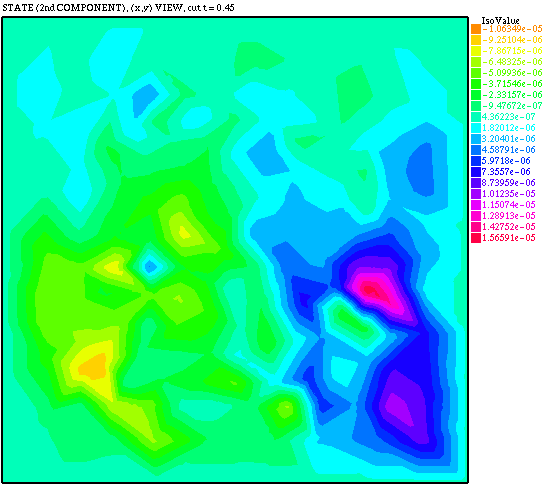}}
    \caption{Heat maps of the second state component at diverse times.}
    \label{fig:State2HMs}
\end{figure}
\FloatBarrier

%% file: Comments.tex
\section{Comments and perspectives} \label{Comments}

\subsection{On the constitutive law for the shear stress}

Upon scrutinizing the proof of Lemmas \ref{HIsWellDefined} and \ref{HIsC1}, we conclude that they still hold for any function $\nu : \mathbb{R}^{N\times N} \rightarrow \mathbb{R}$ in (\ref{StressTensor}) having the following properties:
\begin{itemize}
    \item $\nu \geqslant \nu_0,$ for some constant $\nu_0>0;$
    \item $\nu$ is of class $C^3\left( \mathbb{R}^{N\times N} \backslash \left\{ 0 \right\} \right);$
    \item There exists $r>0$ such that
    $$
    |D^k \nu(A)|\leqslant C\left(1 + |A|^{(r-k)^{+}} \right),
    $$
    for $k=0,1,2,3,$ and for every $A \in \mathbb{R}^{N\times N}\backslash\left\{ 0\right\}.$
\end{itemize}
With Lemmas \ref{HIsWellDefined} and \ref{HIsC1} at hand, we can follow the remaining steps towards the main result, i.e., Theorem \ref{MainThm}, in the same manner as we proceeded in Section \ref{sec3}. This more general class of constitutive laws includes the one determining the reference model of this paper, namely, $\nu(A) := \nu_0 + \nu_1|A|^r,$ when $r\in \left\{ 1, 2 \right\}$ or $r\geqslant 3.$ An example of another class of functions $\nu$ for which the properties we stated above hold are
$$
\nu(A) := \nu_0 \left(1 + \nu_1 |A|^2 \right)^{r/2},\, r \in \left\{1,2\right\}\cup \left[3,\infty\right[.
$$

\subsection{On the use of the gradient instead of the deformation tensor}

We can replace the gradient of the velocity field in (\ref{StressTensor}) with the deformation tensor, $Dy = \left( \nabla y + \nabla y^T\right)/2,$ without losing any of the results we established. From a practical viewpoint, this form of the model is more realistic. Analyzing the estimates we carried out throughout the present work, it is easy to see that the techniques we employed work just as well under this substitution. In particular, we notice the new framework shares the linearization around the zero trajectory with the one we studied in Section \ref{sec2}. Using the estimates developed there, alongside Korn-type inequalities, we can prove all of the corresponding results in Sections \ref{sec3} and \ref{sec4} for this alternate version of the model \eqref{Model}-\eqref{ConstitutiveLaw}. 

\subsection{On extensions of Theorem \ref{MainThm} and some related open questions}

\textbf{Boundary controllability.} We remark that a corresponding boundary local null controllability result follows from Theorem \ref{MainThm}. In effect, let us assume that the initial data $y_0$ belongs to $H^5_0(\Omega)\cap V,$ being sufficiently small in the (strong) topology of this space, and that we act with a control on a smooth portion $\gamma$ of the boundary $\partial \Omega$ (with $\gamma \neq \partial \Omega$ and $\gamma \neq \emptyset$). We can employ standard geometrical arguments to extend $\Omega$ to an open region $\widehat{\Omega},$ with a smooth boundary $\partial\widehat{\Omega},$ and in a way that $\partial \Omega \backslash \gamma \subset \partial \widehat{\Omega}.$ Acting distributively over $\omega:= \widehat{\Omega} \backslash \overline{\Omega},$ with $y_0$ extended to zero outside of $\Omega,$ we obtain a control $\widehat{v} \in L^2(\left]0,T\right[ \times \omega)$ driving the corresponding state $\widehat{y}$ to zero at time $T.$ A boundary control for the original problem is $\widehat{y}|_{\left[0,T\right]\times \gamma}.$

\textbf{Local controllability to trajectories.} Regarding the local exact controllability to trajectories, there are two key aspects to investigate. Firstly, we must prove a global Carleman inequality, analogous to Proposition \ref{Carleman}, but for the adjoint system of the linearization around the given trajectory, cf. Lemma \ref{HIsC1}. Secondly, we have to extend the estimates of Section \ref{sec2} for this linearized problem. These endeavors are not straightforward, whence we leave this question open for future investigations.

\textbf{On the restrictions on the exponent $r.$} We notice that the estimates of Section \ref{sec3} are not immediately extensible for the values of $r > 0$ outside of $\{1,2\}\cup \left[3,\infty\right[.$ However, we conjecture that our main result (viz., Theorem \ref{MainThm}) is still in force for these values of $r.$ A possible way to establish this is to parametrically regularize the function $\nu$ around zero, and attentively keep track of the regularization parameters throughout the estimates. We leave this question open here. 

\textbf{Requirements on the initial datum.} Through another regularization argument, we possibly could require a less restrictive topology for the initial datum in the main theorem. Namely, if we assume $y_0 \in H$ only, we ought to carry out estimates for the uncontrolled problem (corresponding to \eqref{Model} with $v\equiv 0$) to show that there exists $t_0 \in \left]0,T\right[$ for which $\|y(t_0,\cdot)\|_{H^5(\Omega)^N\cap V} \leqslant \eta,$ as long as $\|y_0\|_{H}$ is sufficiently small. We choose not to delve in the technicalities of these estimates here (see \cite[Lemma 5]{coron2016small} for the application of such an argument in the case of the Navier-Stokes equations with the Navier boundary condition). However, we emphasize that this is a non-trivial task. Thus, assuming this is valid, Theorem \ref{MainThm} asserts that there exists a control $v \in L^2(\left]t_0,T\right[\times \omega)$ driving $y(t_0,\cdot)$ to zero at time $T.$ From the exponential decay of solutions, see \cite[Theorem 4.5]{malek1995existence}, this argument immediately provides a large-time global null controllability result.

\textbf{Remarks on other boundary conditions.} We observe that, if instead of no-slip boundary conditions, we assume Navier boundary conditions, the method of \cite{coron2016small}, used for the Navier-Stokes equations, may apply to the current model. If we manage to deal with the additional terms figuring in the expansions we must make after an appropriate time rescaling, especially the boundary layers, we should obtain a small-time global exact controllability to trajectories result (under Navier boundary conditions). Alternatively, if we consider the model \eqref{Model}-\eqref{ConstitutiveLaw} with $\Omega = \mathbb{T}$ (the $N-$dimensional torus) and periodic boundary conditions, then we can easily conduct the regularizing argument for the initial datum we outlined above, whence we can prove large-time global null controllability for this model --- we omit the details here.

\textbf{Stabilization results.} It might be that, for $\nu_1 > 0,$ an appropriate use of the stabilizing effect of the power-law model makes it easier to establish stabilization results for this class of non-Newtonian fluids. In this way, we propose that our current contributions could bridge such results with global null controllability ones. We remark that, even for the Navier-Stokes equations (corresponding to $\nu_1 = 0$) under no-slip boundary conditions, whether global null controllability holds is an open problem. We suggest that such results for \eqref{Model}-\eqref{ConstitutiveLaw} (with $\nu_1 > 0$) could provide insight on this important open question.